\documentclass[a4paper,10pt]{article}
\usepackage[utf8]{inputenc}
\usepackage[T1]{fontenc}
\usepackage[english]{babel} 

\usepackage{mathpazo}
\usepackage{mathptmx}
\usepackage{amsmath}
\usepackage{amssymb}
\usepackage{mathrsfs}
\usepackage{amsthm}
\usepackage{amsfonts}
\usepackage{fancyhdr}
\usepackage[all]{xy}

\numberwithin{equation}{section}
{\theoremstyle {definition} \newtheorem {defi} {Definition} [section] }
{\theoremstyle {plain}  \newtheorem {thm} [defi] {Theorem}}
{\theoremstyle {plain}  \newtheorem {cor} [defi]{Corollary}}
{\theoremstyle {plain} \newtheorem {prop} [defi]{Proposition}}
{\theoremstyle {plain} \newtheorem {nem}[defi] {Lemma}}
{\theoremstyle {remark} \newtheorem {rmq}[defi] {Remark}}

\def\E{{\Bbb{E}}}
\def\T{{\Bbb{T}}}
\def\P{{\Bbb{P}}}
\def\R{{\Bbb{R}}}

\def\N{{\Bbb{N}}}

\def\dt{{\partial_t}}

\def\ggeq{{ \gtrsim }}
\def\lleq{{ \lesssim }}

\def\ds{{\partial^2_{tt}}}

\def\ggeq{{\ \gtrsim\ }}
\def\lleq{{\ \lesssim\ }}
\def\p{{\mathfrak{P}_t^\alpha}}
\def\q{{\mathfrak{P}_t^{\alpha *}}}
\def\b{{\mathfrak{P}_t^k}}
\def\d{{\mathfrak{P}_t^{k *}}}

\title{Invariant measure and large time dynamics of the cubic Klein-Gordon equation in $3D$}
\author{Mouhamadou SY}
\date{\textit{Université de Cergy-Pontoise\\ Laboratoire AGM UMR 8088 CNRS \\ 2 av. Adolphe Chauvin,
95302 Cergy-Pontoise Cedex, France\\ mouhamadou.sy@u-cergy.fr}}

\begin{document}

\maketitle

\begin{abstract}
In this paper we construct an invariant probability measure concentrated on $H^2(K)\times H^1(K)$ for a general cubic Klein-Gordon equation (including the case of the wave equation). Here $K$ represents both the $3$-dimensional torus or a bounded domain with smooth boundary in $\R^3.$ That allows to deduce some corollaries on the long time behaviour of the flow of the equation in a probabilistic sense. We also establish qualitative properties of the constructed measure. This work extends the Fluctuation-Dissipation-Limit (FDL) approach to PDEs having only one (coercive) conservation law.
\end{abstract}
\paragraph{Classification:}28D05, 60H30, 35B40, 35L05, 35L71.
\tableofcontents

\section{Introduction}
The Klein-Gordon (KG) equation 
\begin{equation}\label{KG}
\ds u-\Delta u+m_0^2u+u^3=0, \ \ (t,x)\in\R_+\times K,
\end{equation}
is a model of evolution of a relativistic massive particle. Here $u$ is a real-valued function,  $m_0^2\in \R$ is the square of the mass of the particle and $K\subset\Bbb R^3$ is the physical space. The KG equation is a Hamiltonian PDE, with the Hamiltonian 
\begin{equation}\label{energy}
E(u,\dt u)=\frac{1}{2}\int \left(|\dt u|^2+|\nabla u|^2+m_0^2|u|^2\right)dx+\frac{1}{4}\int u^4dx.
\end{equation}
The natural phase space is then the Sobolev product space $H^1(K)\times L^2(K)$ containing the vectors $y=[u,\dt u]$.

Our purpose is to construct an invariant measure and to study some of its qualitative properties. The motivations of such a problem are discussed below as well as the difficulties of the question in the context of $(\ref{KG})$. Moreover, a panorama of applications coming from general ergodic theorems is presented in Section \ref{sect_ergo}. Here, we consider both the periodic and the bounded domain setting.\\
 Both on $\T^3$ or on a domain $D$ (with boundary conditions $u|_{\partial K}=0$), we denote by $(\lambda_j,e_j)_{j\in\N}$ the couples (eigenvalue, eigenfunction) of the Laplacian operator $-\Delta$. Remark that $\lambda_0=0$ only when the problem is posed on a torus and that, in both cases, $(\lambda_j)_j$ is a sequence of non-negative real numbers increasing to infinity like $j^{\frac{2}{3}}$ (Weyl asymptotics). We define the Sobolev space of order $m\in\R$ by
\begin{equation*}
H^m=\left\{u=\sum_{j=0}^\infty u_je_j:\ \ \|u\|_m^2 :=\sum_{j=0}^\infty(m_0^2+\lambda_j)^mu_j^2<\infty\right\},
\end{equation*}
where $m_0^2>-\lambda_0$.
The space $H^0$ is also denoted by $L^2,$ and $\|.\|_0$ by $\|.\|.$ The inner product on $H^m$ corresponding to the norm $\|.\|_m$ is denoted by $(,)_m$ and $(,)_0$ is simply written $(,)$. We have the following embedding inequality:
\begin{equation}\label{chap5_embed_ineq}
\|u\|_m^2\geq (m_0^2+\lambda_0)^{(m-s)}\|u\|_{s}^2\ \ \ \ for\ \ any \ \ m\geq s \ \ in\ \ \R.
\end{equation}
The product Sobolev space $H^m\times H^{n}$ is denoted by $\mathcal{H}^{m,n}$ and endowed with the norm defined, for any vector $[u,v]\in \mathcal{H}^{m,n},$ by
\begin{equation*}
\|[u,v]\|_{m,n}^2:=\|u\|_m^2+\|v\|_{n}^2,
\end{equation*}
and the corresponding inner product is denoted by
\begin{equation*}
([u_1,v_1],[u_2,v_2])_{m,n}:=(u_1,u_2)_m+(v_1,v_2)_n.
\end{equation*}
Set $\Delta_0:=\Delta-m_0^2$, then the equation and its Hamiltonian are rewritten as
\begin{align*}
\ds u-\Delta_0u+u^3=0,
\end{align*}
\begin{equation*}
E(u,\dt u)=\frac{1}{2}\|[u,\dt u]\|_{1,0}^2+\frac{1}{4}\int u^4.
\end{equation*}
Notice that
\begin{align*}
\|u\|_m=\|(-\Delta_0)^{m/2}u\|, \ \ \ (u,v)_m=((-\Delta_0)^{m/2}u,(-\Delta_0)^{m/2}v).
\end{align*}

\subsection{Invariant measures for PDEs: Motivations and approaches}
Solving the Cauchy problem for a PDE is equivalent to specifying a phase space $E$ and a (semi-) group of continuous maps $\phi_t:E\to E$ which governs the evolution in time of the phase-vectors. The couple $(E,\phi_t)$ defines a dynamical system. One of the important questions in qualitative theory of PDEs is to describe the long time behavior of $\phi_t.$ A Borel measure $\mu$ on $E$ is called invariant for $\phi_t$ if for any Borel set $\Gamma\subset E$, for any $t$, we have
\begin{equation*}
\phi_{t*}\mu(\Gamma):=\mu(\phi_t^{-1}(\Gamma))=\mu(\Gamma).
\end{equation*}
Existence of such a measure allows to draw some conclusions on long time properties for the system $(E,\phi_t)$ (see e.g. Birkhoff, Poincaré and von Neumann theorems in Section $\ref{sect_ergo}$). The concept of invariant measure plays also an important role in probabilistic global wellposedness\footnote{Let us mention the paper by Burq and Tzvetkov \cite{burktzvt_probwav} introducing new approach to study probabilistic wellposedness which does not use an invariance property.} for PDEs by providing a way to control globally the induced flow.

There are, at least, two approaches to construct invariant measures; for finite-dimensional equations representing  the evolution of a divergence free vector-field, the so-called Liouville theorem states that the Lebesgue measure defined on the associated phase space is preserved along the time. This result covers indeed the finite-dimensional Hamiltonian flows and their theory of Gibbs measures. The question of infinite-dimensional Gibbs measures (for Hamiltonian PDEs) is not directly implied by this general theorem, but is studied in many works with its help.\\ 
The other result is given by the Krylov-Bogoliubov theorem for dynamical systems under some compactness assumptions. A method has been developed with use of this argument to approach more general PDEs. \\
 Let us briefly present the general philosophy of two approaches of the PDEs invariant measures problem and compare them on some of their characteristic points. 
\paragraph{\textbf{Gibbs measures theory for PDE.}}
For a PDE having a "nicely structured" conservation law $E(u)$, we can expect that, under proper definition, the expression $"e^{-E(u)}du"$ could be an invariant measure. An approach consists in projecting the PDE on finite dimensional subspaces of increasing dimension. Then a sequence of ordinary differential equations are considered and the idea is to use the Liouville theorem. We get, then, a sequence (w.r.t. the dimension) of invariant measures (having a density w.r.t. Gaussian measures). An accumulation point is the measure we look for.
\paragraph{\textbf{FDL measures theory.}}The Fluctuation-Dissipation-Limit approach consists in approximating the Hamiltonian dynamics by some kind of "compact" ones. The Krylov-Bogoliubov theorem provides then a sequence of invariant measures whose accumulation point could be invariant for the limiting equation. Namely, a damping term (given by a negative operator) and a stochastic forcing are added to the equation. The former should give the compactness in question while the latter is intended to maintain the evolution that the damping tends to attenuate:
\begin{equation*}
"PDE=\alpha\text{Damping}+s(\alpha)\text{(Forcing)}".
\end{equation*}
The function $s$ will be chosen so that there will be a balance between the contributions of the added two terms and to ensure then the tightness of the sequence of constructed invariant measures in order to get the existence of the desired measure. Here again, a leading role is played by conservation laws. 
\paragraph{\textbf{Gibbs measures vs FDL measures.}}
The first remarkable difference between the two approaches is that Gibbs measures reduce the regularity of the underlying conservation law, that is, their supports are less regular than the conservation law used in the construction (with reduction of $1/2+$), this fact imposes systematically a \textit{threshold of regularity} to the support. Whereas the FDL measures increase (by $1$, if damped by $\Delta$) the initial regularity; the construction does not impose directly a threshold on the "living space" of these measures. This makes the Gibbs measures particularly adapted to approach some spaces of low regularity and to give a probabilistic alternative to the Cauchy theory for PDEs. However, FDL measures can approach some high regularity spaces, seemingly inaccessible by the formers, to establish long time behavior properties of PDEs (see \cite{sykg}). The intermediate situation is common to both.\\
The second fact is that Gibbs measures enjoy many good properties being of Gaussian type, while in the case of FDL measures no qualitative property is directly deduced. However, some stochastic methods are developed in \cite{kuksin_nondegeul,armen_nondegcgl,KS12} to investigate non-degeneracy properties.

For Klein-Gordon related equations, Gibbs measures are constructed both in finite or in infinite volumes, see for instance \cite{burktzvet3Dinvgibb, bourbulut, AdSkgnonpdqset,xu2014gibbsmNLWinfvol}.  These measures concern radial solutions (in the $3D$ case) and are then concentrated on $\mathcal{H}^{1/2-,-1/2-}$.  The question of non radial Gibbs type measure for the three-dimensional Klein-Gordon equation encounters an obstruction. Indeed, such a measure has to be defined on $\mathcal{H}^{-1/2-,-3/2-}$ where the nonlinearity would become problematic. In contrast with the loss of regularity inherent to the Gibbs measure approach, the FDL method proceeds by regularization. In that approach, the nonlinearity is still tractable even in a non radial context. However, as we will see it later on, the uniqueness of a coercive conservation law gives rise to some difficulties in the method.
\subsection{Statement of the main result and comments}
To present the main result of the paper, recall that $\lambda_0$ denotes the first eigenvalue of $-\Delta$ in both settings considered in this work.
\begin{thm}
Let $m_0^2>-\lambda_0$, then, in both settings, there is an invariant measure $\mu$ for $(\ref{KG})$ defined on $\mathcal{H}^{1,0}$ and satisfying:
\begin{itemize}
\item \begin{equation*}
\mu(\mathcal{H}^{2,1})=1;
\end{equation*}
\item \begin{equation*}
0<\int_{\mathcal{H}^{1,0}}\|y\|_{2,1}^2\mu(dy)<\infty;
\end{equation*}
\item there is $\sigma>0$ such that
\begin{equation*}
\int_{\mathcal{H}^{1,0}}e^{\sigma E(y)}\mu(dy)<\infty,
\end{equation*}
consequently $\mu$ enjoys a Gaussian control property w.r.t. the norm $\mathcal{H}^{1,0};$
\item the distribution under $\mu$ of the Hamiltonian $E(y)$ has a density w.r.t. the Lebesgue measure on $\R.$
\end{itemize}
\end{thm}
\begin{rmq}
In fact we have a family of invariant measures for $(\ref{KG})$ on $\mathcal{H}^{2,1}$, one can see that after parametrising the diffusion constants associated to the approximation problem $(\ref{ddKG}).$ 
\end{rmq}
The Poincaré recurrence theorem (see Section \ref{sect_ergo}) implies
\begin{cor}
For $\mu$-almost any $y=[u,v]$ in $\mathcal{H}^{2,1}$, there is a sequence $t_k$ going to infinity as $k\to\infty$ such that
\begin{equation*}
\lim_{k\to\infty}\|\phi_{t_k}y-y\|_{2,1}=0.
\end{equation*}
Here $\phi_t$ denotes the flow of $(\ref{KG})$ on $\mathcal{H}^{2,1}.$
\end{cor}

Let us make some comments on the results. First, remark that by Sobolev embedding, the solutions concerned by our results are, in particular, continuous in the $x$ variable. Second, in the case where the equation is posed on a bounded domain, $\lambda_0$ is positive; then the massless case, i.e. the wave equation, is covered by our result. Moreover, $m_0$ is also allowed to be an imaginary number, in that situation $(\ref{KG})$ is associated to a particle with imaginary mass. Such hypothetical particles, named tachyons, are used in some areas of theorical physics.

To obtain these results, an additional difficulty compared to the earlier works is the fact that we know only one coercive conservation law for KG (in the case of the torus we have also the momentum which is not coercive), that implies a "lack of estimates". Notice also that the FDL approach was developped for Hamiltonian PDEs having at least two "good" conservation laws \cite{kuk_eul_lim,KS04,KS12,sykg}. The present paper is also intended to extend this approach to Hamiltonian PDEs having one conservation law.\\
In order to confront the "lack of conservation" present in our context, we introduce what we call \textit{almost conservation laws} associated to $(\ref{KG})$. These quantities play essentially the same role  in the construction of an invariant measure as the two conservation laws, however, they cannot be used in studying  its qualitative properties. That problem is solved by an approximation argument combined with the approach of $\cite{armen_nondegcgl,KS12}$.\\ Notice that the concept of "almost conservation laws" is the main ingredient in the so-called \textit{I}-method technique, overcoming the lack of conservation in the study of wellposedness and asymptotic behavior of dispersive PDEs (see e.g. \cite{ckstt,tao}). However, while in the $I-$method theory the modification consists in damping the high frequencies, in our situation we opt for an additive regular perturbation which accommodates better with our damping scheme. In Section \ref{sect3} we define precisely our understanding of that concept, then we introduce two of such quantities and derive their respective dissipation rates whose statistical control along the time and the viscosity parameter takes the central place in our analysis. Notice also that an argument of modification of energy was developed in \cite{nikmqi,ohtv} in the context of quasi-invariant measures theory for Hamiltonian PDEs. \\
We now describe the Fluctuation/Dissipation scheme that we apply to the Klein-Gordon equation in our work. Consider
the stochastic PDE
\begin{equation}\label{ddKG}
\ds u-\Delta_0 u+u^3=\alpha\Delta_0\dt u+\sqrt{\alpha}\eta,
\end{equation}
where 
\begin{equation*}
\eta(t,x)=\frac{d}{dt}\zeta(t,x)=\frac{d}{dt}\sum_{m=0}^\infty a_me_m(x)\beta_m(t).
\end{equation*}
Here $\beta_m$ are independent standard Brownian motions and $a=(a_m)$ is a sequence of  real numbers. For $n\geq 0,$ define the number
\begin{equation*}
A_n=\sum_{m=0}^\infty a_m^2\lambda_m^{n}.
\end{equation*}
The vector $y_t=[u,\dt{u}]$ is a random variable on a complete probability space $(\Omega,\mathcal{F},\P)$ with range in Sobolev spaces. We assume that the filtration $\mathcal{F}_t$ associated to $\zeta_t$ is right continuous and augmented w.r.t. $(\mathcal{F},\P).$\\
For given positive quantities $A,B$ satisfying
$A\leq c_1B$, we write $A\lleq B$ or $A\lleq_{c_1}B.$ The vectors in $H^m\times H^n=:\mathcal{H}^{m,n}$ are denoted with the symbol $[,]$ while the symbol $(,)$ represents the inner product in $L^2$. 

\section{Ergodic theorems and some consequences}\label{sect_ergo}
In this section we discuss some details about the PDE's motivations of invariant measures theory via some general results from ergodic theory. We can also see the introduction of \cite{thomann}.
\subsection{Ergodic theorems}
Consider the measurable dynamical system $(X,\phi_t,\mu)$ constructed from an evolution equation, here the probability measure $\mu$ is invariant under the flow $\phi_t$. In the case of a reversible dynamics (e.g. Hamiltonian equations), the transformations $(\phi_t)_{t\in\R}$ form a group and $\phi_t^{-1}=\phi_{-t}$, we adopt this hypothesis in the present section altough all the results we are discussing here can be adapted to the semi-group case by classical ways. In \cite{koopman}, Koopman observes that the (a priori) nonlinear transformations $(\phi_t)$ induce linear ones on the space $L^2(X,\mu)$. These induced transformations $U_t$ are defined for any function $f:L^2(X,\mu)\to\R$ by
\begin{equation*}
U_tf(w)=f(\phi_tw) \  \ \forall w\in X.
\end{equation*}

The linearity and group property of $(U_t)$ are clear, and for any $t\in\R,$ $U_t$ defines an isometry on $L^2(X,\mu)$. In fact
\begin{equation*}
\|U_tf\|_{L^2}^2=\int_X|U_tf(w)|^2\mu(dw)=\int_X|f(\phi_tw)|^2\mu(dw).
\end{equation*}
A standard approximation (by simple functions) argument combined with the invariance of $\mu$ establishes the desired property. We also remark that $U_t^{-1}=U_{-t}.$ 
A message contained in Koopman's observation is that, provided that an invariant measure is given, the "nonlinear description" of the evolution of the states can be replaced by a "linear description" on the observables. We then pass from a nonlinear "microscopic" study to a linear "macroscopic" one. In the latter setting, general theorems such as Von Neumann and Birkhoff ergodic theorems, can be used to obtain some statistical properties of the dynamics. Let us present a version of these theorems (for their proofs and more results concerning them see \cite{krengel,coudene}). 
Let $T>0$, set the Birkhoff average
\begin{equation*}
S_Tf(w)=\frac{1}{T}\int_0^TU_tf(w)dt,
\end{equation*}
and the following invariants of the evolution
\begin{align*}
I_1 &=\{h\in L^2(X,\mu):\ U_th=h,\ \forall t\},\\
I_2 &=\{A\ \in \text{Bor}(X):\ \phi_t^{-1}A=A,\ \forall t\},
\end{align*}
where Bor$(X)$ is the Borel $\sigma-$algebra of $X.$
$I_1$ and $I_2$ are related by the fact that
$$A\in I_2\Leftrightarrow \Bbb 1_A\in I_1.$$
\begin{thm}[Von Neumann]
For all $f\in L^2(X,\mu)$, we have, as $T\to\infty$, 
\begin{equation*}
S_Tf\to P_{I_1}f \ \ \text{in $L^2(X,\mu)$},
\end{equation*}
where $P_{I_1}$ denotes the orthogonal projection onto $I_1.$
\end{thm}
\begin{thm}[Birkhoff]
For all $f\in L^1(X,\mu)$, we have as $T\to\infty$ 
\begin{equation*}
S_Tf\to \E_{I_1}f \ \ \text{in $L^1(X,\mu)$},
\end{equation*}
where $E_{I_2}$ denotes the conditional expectation w.r.t. $I_2.$ This convergence holds also $\mu-$almost surely on $X$.
\end{thm}
\subsection{Consequences}\label{chap4_erg}
This subsection is an "adaptation" to continuous dynamical systems of the idea contained in \cite{coudene} (Chapter $1$, Exercise $9$), one can also see the proof given in \cite{taoerg} and the discussion of \cite{thomann}.\\  
Let $A$ and $B$ be two Borel sets in $X$, $1_A$ and $1_B$ are the indicator functions of $A$ and $B$ respectively, we denote the orthogonal projection onto $I_1$ just by $P$. We have
\begin{prop}
\begin{equation}
\frac{1}{t}\int_0^t\mu(A\cap\phi_s^{-1}B)ds\to\langle P1_A,P1_B\rangle \ \ \text{as $t\to\infty$}.\label{conseq1}
\end{equation}
In particular,
\begin{equation}
\frac{1}{t}\int_0^t\mu(A\cap\phi_s^{-1}A)ds\to\|P1_A\|^2_{L^2} \ \ \text{as $t\to\infty$},\label{conseq2}
\end{equation}
and 
\begin{equation}
\mu(A)^2\leq \lim_{t\to\infty}\frac{1}{t}\int_0^t\mu(A\cap\phi_s^{-1}A)ds\leq \|P1_A\|_{L^2(X)}\sqrt{\mu(A)}.\label{conseq3}
\end{equation}
\end{prop}
\begin{proof}
We prove $(\ref{conseq2})$ by taking $B=A$ in $(\ref{conseq1})$. Now $(\ref{conseq3})$ is derived from $(\ref{conseq2})$ as follow
\begin{align*}
\mu(A)=\E_\mu1_A=\E_\mu\E_{I_1}(1_A)=\E_\mu(P1_A)=\|P1_A\|_{L^1}\leq \|P1_A\|_{L^2}.
\end{align*}
On the other hand, we use the property $P^*P=P^2=P$ and the Cauchy-Schwarz inequality to find 
\begin{align*}
\|P1_A\|_{L^2}^2\leq \|P1_A\|_{L^2(X)}\sqrt{\mu(A)}.
\end{align*}
It remains to prove $(\ref{conseq1}).$ To this end, we use the von Neumann ergodic theorem, which establishes convergence in the $L^2$-norm. It follows that we also have weak convergence, and therefore, as $t\to\infty,$,
\begin{align*}
 \langle 1_A,\frac{1}{t}\int_0^t1_B(\phi_s)ds\rangle \rightarrow\langle 1_A,P1_B\rangle=\langle P1_A,P1_B\rangle,
\end{align*}
but
\begin{align*}
 \langle 1_A,\frac{1}{t}\int_0^t1_B(\phi_s)ds\rangle&=\frac{1}{t}\int_0^t\langle 1_A,1_B(\phi_s)\rangle ds\\
&=\frac{1}{t}\int_0^t\langle 1_A,1_{\phi_s^{-1}(B)}\rangle ds.
\end{align*}
Now it is clear that
\begin{align*}
\langle 1_A,1_{\phi_s^{-1}(B)}\rangle =\mu(A\cap\phi_s^{-1}B).
\end{align*}
That finishes the proof.
\end{proof}
We have the quantitative version of the Poincaré recurrence theorem 
\begin{prop}[Poincaré recurrence theorem]
\begin{equation}
\overline{\lim}_{t\to+\infty}\mu(A\cap\phi_t^{-1}A)\geq\mu(A)^2.\label{lim_sup_poinca_rec}
\end{equation}
Accordingly, if $\mu(A)>0,$ then $A\cap\phi_{t_k}A$ is non empty for a sequence $(t_k)$ converging to infinity with $k.$
\end{prop}
\begin{proof}
For $t>1,$ we write
\begin{align*}
\frac{1}{t^2}\int_0^{t^2}\mu(A\cap\phi_s^{-1}A)ds &\leq\frac{1}{t^2}\int_0^{t}\mu(A\cap\phi_s^{-1}A)ds+\frac{1}{t^2}\int_{t}^{t^2}\sup_{s\geq t}\mu(A\cap\phi_s^{-1}A)ds\\
&\leq\frac{1}{t}+\frac{t-1}{t}\sup_{s\geq t}\mu(A\cap\phi_s^{-1}A).
\end{align*}
Passing to the limit $t\to\infty$ and using the left-hand inequality in $(\ref{conseq3}),$ we obtain $(\ref{lim_sup_poinca_rec})$.
\end{proof}
\section{Almost conservation laws for KG}\label{sect3}
In the context of the FDL approach, there is some kind of "algebraic structure" that a functional in hand has to respect to be fruitful in the construction of an invariant measure. Indeed, one needs uniform in $\alpha$ controls in the passage to the limit from the stochastic model towards  the Hamiltonian PDE. In the estimation procedure, terms interacting with the damping (of order $\alpha$) are added with terms interacting with the forcing (of order $\alpha$ after taking the quadratic variation) and "order $1$" terms. To get the needed uniformity, the order $1$ terms must vanish under expectation w.r.t. an invariant measure. In the case of a conservation law, this requirement is satisfied because of the special "algebraic relations" that the latter shares with the equation. We call almost conservation law any (non preserved) functional that satisfies this requirement. Such a functional must depend on the damping model. Now we state a precise definition of our understanding of almost conservation law:\\
Consider a PDE
\begin{equation}\label{chp5_PDE_gen}
\dt u=f(u),
\end{equation}
and a functional $V(u)$, then, formally, we have for any solution $u$ that
\begin{equation*}
\dt V(u)=(\nabla_u V(u),f(u)).
\end{equation*}
We call the quantity $(\nabla_u V(u),f(u))$ by \textit{the evolution rate} of $V$ under the equation $(\ref{chp5_PDE_gen})$. It is clear that this term is zero iff $V$ is a conservation law for this equation. Now consider a linear perturbation of $(\ref{chp5_PDE_gen})$
\begin{equation}\label{chap5_PDE_gen_diss}
\dt u=f(u)+\alpha Lu,
\end{equation}
then
\begin{equation*}
\dt V(u)=(\nabla_u V(u),f(u))+\alpha(\nabla_uV(u),Lu).
\end{equation*}
In the case where $V$ is a conservation law for the equation, then
$\dt V(u)$ is of "order" $\alpha$, $i.e.$ $\dt V(u)$ is of the form $\alpha h(u)$ where $h$ does not depend on $\alpha.$
\\
A functional $V$ is called almost conservation law for $(\ref{chp5_PDE_gen})$ relatively to $(\ref{chap5_PDE_gen_diss})$ if 
\begin{itemize}
\item $V$ is not a conservation law,
\item for a solution $u$ to $(\ref{chap5_PDE_gen_diss})$, $\dt V$ remains of order $\alpha.$
\end{itemize}

\begin{rmq}
The evolution rate $\dt V$ for an almost conservation law $V$ must vanish when $\alpha=0,$ therefore $V$ has to be a perturbation of a conservation law.
\end{rmq}

In the present work, the damping scheme for $(\ref{KG})$ we consider is
\begin{equation}\label{chap_4_equ_san_eta}
\ds u-\Delta_0u+u^3=\alpha\Delta_0\dt u,\ \ \ \alpha\in (0,1).
\end{equation}
Let us introduce the following quantities:
\begin{align*}
G_1(y) &=E(y)+\frac{\alpha(m_0^2+\lambda_0)}{2}\int u\dt u+\frac{\alpha^2(m_0^2+\lambda_0)}{4}\|u\|_1^2,\\
G_2(y)&=E(y)-\frac{\alpha}{2}\int\dt u\Delta_0u+\frac{\alpha^2}{4}\|u\|_2^2.
\end{align*}
With use of $(\ref{chap5_embed_ineq})$ (for $m=1$, $s=0$) and $y$ is the vector $[u,\dt u]$. We remark that
\begin{align}
G_1(y) &\geq E(y)-\frac{\alpha^2(m_0^2+\lambda_0)^2}{4}\|u\|^2-\frac{1}{4}\|\dt u\|^2+\frac{\alpha^2(m_0^2+\lambda_0)}{4}\|u\|_1^2\geq E(y)-\frac{1}{4}\|\dt u\|^2\geq \frac{1}{4}E(y),\label{G_1pgE1}\\
G_2(y) &\geq E(y)-\frac{\alpha^2}{4}\|\Delta_0u\|^2-\frac{1}{4}\|\dt u\|^2+\frac{\alpha^2}{4}\|u\|_2^2=E(y)-\frac{1}{4}\|\dt u\|^2\geq \frac{1}{4}E(y).
\end{align}
Hence, in particular, the positivity of $G_1$ and $G_2.$

We have that the functionals $G_1$ and $G_2$ are almost conservation laws for $(\ref{KG})$ relatively to our dissipation scheme. The following controls (obtained in Proposition $\ref{prop_estimat_genral}$) express this fact:
\begin{align}
G_1(y_t)+\alpha\int_0^tL_1(y_s)ds &\leq G_1(y_0),\label{cont_almost}\\
 G_2(y_t)+\alpha\int_0^tL_2(y_s)ds &\leq G_2(y_0)+\alpha C\int_0^t\|u\|_{L^6}^6ds,
\end{align}
where $C$ is a constant independent of $\alpha,$ and we set
\begin{align*}
L_1(y)&=\frac{1}{2}\left\{(m_0^2+\lambda_0)\|u\|_1^2+2\|\dt u\|_1^2-(m_0^2+\lambda_0)\|\dt u\|^2+(m_0^2+\lambda_0)\|u\|_{L^4}^4\right\},\\
L_2(y)&=\frac{1}{2}\left(1^-\|u\|_2^2+\|\dt u\|_1^2\right),
\end{align*}
and $1^-=1-\epsilon$, with $\epsilon>0$ arbitrarily close to $0.$ \\ 
We give some useful estimates: 
\begin{prop} For all $[u,v]\in \mathcal{H}^{1,1},$ we have
\begin{align}
G_1(u,v) &\leq \frac{2+m_0^2+\lambda_0}{2\kappa^2}L_1(u,v), \label{E<F_1<E}\\
G_2(u,v) &\leq \frac{5E(u,v)+L_2(u,v)}{4},\label{ineq_G_2>half_E}
\end{align}
where $\kappa=\min(m_0^2+\lambda_0,1)$.
\end{prop}
The proof of the above proposition is straightforward.\\
Denoting by $\gamma_0$ the positive number $2\kappa^2/(2+m_0^2+\lambda_0)$, we infer from $(\ref{cont_almost})$ and $(\ref{E<F_1<E})$ that, for any solution $y_t\in \mathcal{H}^{1,0}$ to $(\ref{chap_4_equ_san_eta})$, we have
\begin{equation*}
G_1(y_t)\leq e^{-\gamma_0\alpha t}G_1(y_0).
\end{equation*}
Taking this inequality to the power $p>0,$ we get
\begin{equation*}
G_1^p(y_t)\leq e^{-p\gamma_0\alpha t}G_1^p(y_0).
\end{equation*}
Combining this with the embedding $H^1\subset L^6$, we obtain
\begin{equation*}
G_2(y_t)+\alpha\int_0^tL_2(y_s)ds \lleq G_2(y_0)+\alpha G_1(y_0)^3.
\end{equation*}

\begin{prop}
We have the inequalities
\begin{align}
L_1(u,v) &\geq \frac{\kappa}{2}\left(\|[u,v]\|_{1,1}^2+\|u\|_{L^4}^4\right), \label{chap4_L1_controle}\\
L_2(u,v) &\geq \frac{\delta}{2}\|[u,v]\|_{2,1}^2\ \ \ for\ any \ \delta<1.  \label{chap4_L2_controle}
\end{align}
\end{prop}
\begin{proof}
The bound $(\ref{chap4_L2_controle})$ is straightforward and  $(\ref{chap4_L1_controle})$ is obained with use of the lemma $\ref{neem}$ below where we take $m=1,\ s=0$.
\end{proof}
\begin{nem}\label{neem}
For $w\in H^m$ with $m\in\R,$ we have for any $s\leq m$
\begin{equation}
\|v\|_m^2-\frac{(m_0^2+\lambda_0)^{m-s}}{2}\|v\|_{s}^2\geq \frac{1}{2}\|v\|_m^2.\label{chap_4_inegetoil}
\end{equation}
\end{nem}
\begin{proof}
Using $(\ref{chap5_embed_ineq})$, we have 
\begin{equation*}
\|v\|_m^2=\frac{1}{2}\|v\|_m^2+\frac{1}{2}\|v\|_m^2\geq \frac{1}{2}\|v\|_1^2+\frac{(m_0^2+\lambda_0)^{m-s}}{2}\|v\|_s^2.
\end{equation*}
That finishes the proof.
\end{proof}

\section{Global wellposedness for the damped KG}
In this section we consider the following equation
\begin{equation}\label{nonlinear_flow}
\ds v-\Delta_0 v+(v+f)^3=\alpha\Delta_0\dt v,
\end{equation}
where $f$ satisfies
$$\sup_{t\in[0,T]}\sup_{x\in K}|f(t,x)|<\infty\ \ \ \forall T>0, \ \ \ K=D\ or \ \T^3.$$
\subsection{A-priori analysis}
\textbf{Choice of the spaces.} Let us, first, consider the damped linear equation 
\begin{equation}\label{chap4_damped_lin_KG}
\ds v-\Delta_0 v=\alpha\Delta_0\dt v.
\end{equation}
 Both on the periodic or the bounded domain setting, the non-damped equation $(\alpha=0)$ preserves the following quantities:
\begin{equation*}
M_m(r,s)=\frac{1}{2}\|[r,s]\|_{m,m-1}^2 \ \ m=1,2.
\end{equation*}
Now, let us introduce the following "perturbed" versions:
\begin{align*}
N_1(r,s) &=\frac{1}{2}\|[r,s]\|_{1,0}^2+\frac{\alpha(m_0^2+\lambda_0)}{2}\int rs+\frac{\alpha^2(m_0^2+\lambda_0)}{4}\|r\|_1^2,\\
N_2(r,s) &=\frac{1}{2}\|[r,s]\|_{2,1}^2-\frac{\alpha(m_0^2+\lambda_0)}{2}\int s\Delta_0r+\frac{\alpha^2(m_0^2+\lambda_0)}{4}\|r\|_2^2.
\end{align*}
By a standard procedure, we see that a solution $[v,\dt v]$ of the damped linear equation $(\ref{chap4_damped_lin_KG})$ satisfies the following dissipation estimates:
\begin{align*}
N_1(v,\dt v)+ &\frac{\alpha}{2}\int_0^t\left\{(m_0^2+\lambda_0)\|v\|_1^2+2\|\dt v\|_1^2-(m_0^2+\lambda_0)\|\dt v\|^2\right\}ds=N_1(v(0),\dt v(0)),\\
N_2(v,\dt v)+ &\frac{\alpha}{2}\int_0^t\left\{(m_0^2+\lambda_0)\|v\|_2^2+2\|\dt v\|_2^2-(m_0^2+\lambda_0)\|\dt v\|_1^2\right\}ds=N_2(v(0),\dt v(0)).
\end{align*}
Then we use the inequality $(\ref{chap_4_inegetoil})$ to infer the following controls:
\begin{align*}
N_m(v,\dt v) &+\frac{\alpha\kappa}{2}\int_0^t\|[v,\dt v]\|_{m,m}^2ds\leq N_m(v(0),\dt v(0)) \ \ \ \ m=1,2.
\end{align*}
In view of these estimates, the natural spaces for studying wellposedness of $(\ref{ddKG})$ are 
\begin{equation*}
Z_{m}^T=C([0,T),\mathcal{H}^{m,m-1})\cap L^2_{loc}([0,T),\mathcal{H}^{m,m})\ \ \ for \ \ T\in (0,+\infty],
\end{equation*}
endowed with the norm defined by
\begin{equation*}
\|[u,v]\|_{Z_m^T}=\sup_{t\in[0,T)}\left(\|[u,v]\|_{m,m-1}^2+\alpha\kappa\int_0^t\|[u,v]\|_{m,m}^2ds\right)^{\frac{1}{2}}.
\end{equation*}

\paragraph{Definitions.}\label{intro_section_convol_stoch_et_Ito_lemma}
\begin{defi}\label{intro_SGWP}
 The equation $(\ref{ddKG})$ is said to be stochastically (globally) well-posed in $\mathcal{H}^{m,m-1}$ if for all $T>0$ 
\begin{enumerate}
\item  for any random variable $u_0$ in $\mathcal{H}^{m,m-1}$ which is independent of $\mathcal{F}_t,$ we have, for almost all $\omega\in\Omega$,
\begin{enumerate}
\item (Existence) there exists $u:=u^\omega \in \Lambda_T:=C(0,T;\mathcal{H}^{m,m-1})\cap L^2(0,T;\mathcal{H}^{m,m})$ satisfying the following relation in $\mathcal{H}^{m,m-2}:$
\begin{equation}
[u,\dt u]=[u_0,\dt u|_0]+\int_0^t[\dt u, \Delta_0u-u^3+\alpha\Delta_0\dt u]ds+[0,1]\zeta(t)\ \ \text{for all $t\in[0,T]$},\label{intro_sol_mild}
\end{equation}
we denote this solution by $y(t,u_0):=y^\omega(t,y_0),$ where $y_0$ is the initial vector data.
\item (Uniqueness) if $y_1,y_2\in \Lambda_T$ are two solutions in the sense of $(\ref{intro_sol_mild}),$ then $y_1\equiv y_2$ on $[0,T],$ 
\end{enumerate}
\item (Continuity w.r.t. initial data) for almost all $\omega,$ we have
\begin{equation}
\lim_{y_{0}\to u_{0}'}y(.,u_0)=y(.,u_0') \ \ \ \text{in $\Lambda_T$},
\end{equation}
here $y_0$ and $y_0'$ are deterministic data;
\item the process $(\omega,t)\mapsto y^\omega(t)$ is adapted to the filtration $\sigma(y_0,\mathcal{F}_t)$.
\end{enumerate}
\end{defi}
\
Let $V\subset H\subset V^*$ be three separable Hilbert spaces, with densely embeddings and where $V^*$ is the dual of $V$ w.r.t. $H$. Then $(V,H,V^*)$ is called a Gelfand triple.
\begin{defi}\label{intro_WS}
We say that the equation $(\ref{ddKG})$ has the Ito property on the Gelfand triple $(\mathcal{H}^{m,m-2},\mathcal{H}^{m,m-1},\mathcal{H}^{m,m})$ if
\begin{enumerate}
\item it is stochastically wellposed on $H$;
\item  
the process $h:=[\dt u,\Delta_0u-u^3+\alpha\Delta_0\dt u]$ is $\mathcal{F}_t$-adapted and
\begin{align}
\P\left(\int_0^t(\|y_s\|_{m,m}^2+\|h_s\|_{m,m-2}^2)ds<\infty,\ \ \forall t>0\right)=1,\ \
\|\zeta(t)\|_{m,m-1}<\infty.\label{intro_hyp_KS}
\end{align}
\end{enumerate}
\end{defi}
To such an Ito process we can apply an Ito formula proved in  Section $A.7$ (Theorem $A.7.5$ and Corollary $A.7.6$) of \cite{KS12}.

\paragraph{A-priori estimates for the nonlinear equation $(\ref{nonlinear_flow})$.} 

\begin{prop}\label{prop_estimat_genral} Set $\gamma_1=1+\alpha\frac{m_0^2+\lambda_0}{2}.$
For any solution $q_t=[v_t,\dt{v}_t]$ to $(\ref{nonlinear_flow})$ starting at $q_0=[v_0,\dt{v}_0]$ with $G_1(q_0)<\infty$ and $G_2(q_0)<\infty$, we have 
\begin{align}\label{ap_est_Energy}
G_1(q_t)+\alpha\int_0^tL_1(q_s)ds\leq e^{\gamma_1\int_0^tR(f)ds}\left(G_1(q_0)+\gamma_1\int_0^t\|f\|_{L^6}^4ds\right),
\end{align}
 \begin{equation}\label{ap_est_G_2}
 G_2(q_t)+\alpha\int_0^tL_2(q_s)ds\leq e^{\int_0^tR(f)ds}\left(G_2(q_0)+\frac{1}{4}\int_0^t(2\|f\|_{L^6}^4+\alpha C\|v+f\|_{L^6}^6)ds\right),
 \end{equation}
 where $R(f)=2(24(\|f\|_{L^\infty}+\|f\|_{L^\infty}^2)+\|f\|_{L^6}^2)$ and $C$ is universal.
\end{prop}
\begin{proof}
Rewrite $(\ref{nonlinear_flow})$ into
\begin{align*}
\ds v-\Delta_0 v+v^3=\alpha\Delta_0\dt v-3v^2f-3vf^2-f^3,
\end{align*}
and $G_1(q)$ as $E(q)+I_\alpha(q).$
Since $E(q)$ is preserved by KG, we have
\begin{align*}
\dt E(q) &=-(\dt v,3v^2f+3vf^2+f^3)+\alpha(\dt v,\Delta_0\dt v)\\
&=-3(f\dt v,vf+v^2)-(\dt v,f^3)-\alpha\|\dt v\|_1^2\\
&\leq 3(\|f\|_{L^\infty}+\|f\|_{L^\infty}^2)\|\dt v\|(\|v\|+\|v\|_{L^4}^2)+\|\dt v\|\|f\|_{L^6}^3-\alpha\|\dt v\|_1^2\\
&\leq 3(\|f\|_{L^\infty}+\|f\|_{L^\infty}^2)\left(\frac{1}{2}\|\dt v\|^2+2\|v\|^2+2\|v\|_{L^4}^4\right)+\frac{1}{2}\|\dt v\|^2\|f\|_{L^6}^2+\frac{1}{2}\|f\|_{L^6}^4\\
&-\alpha\|\dt v\|_1^2\\
&\leq (24(\|f\|_{L^\infty}+\|f\|_{L^\infty}^2)+\|f\|_{L^6}^2)E(q)+\frac{1}{2}\|f\|_{L^6}^4-\alpha\|\dt v\|_1^2.
\end{align*}
Now
\begin{align*}
\dt I_\alpha(q) &= \frac{\alpha(m_0^2+\lambda_0)}{2}(\dt v,\dt v)+\frac{\alpha^2(m_0^2+\lambda_0)}{4}\dt\|v\|_1^2\\
&+\frac{\alpha(m_0^2+\lambda_0)}{2}(v,\Delta_0v-(v+f)^3+\alpha\Delta_0\dt v))\\
&=\frac{\alpha(m_0^2+\lambda_0)}{2}\|\dt v\|^2+\frac{\alpha^2(m_0^2+\lambda_0)}{4}\dt\|v\|_1^2-\frac{\alpha(m_0^2+\lambda_0)}{2}(\|v\|_1^2+\|v\|_{L^4}^4)\\
&-\frac{\alpha^2(m_0^2+\lambda_0)}{4}\dt\|v\|_1^2\\ &+\underbrace{\frac{\alpha(m_0^2+\lambda_0)}{2}(vf,3v^2+3vf+f^2)}_{A}.
\end{align*}
Let us notice that
\begin{align*}
A &\leq\frac{\alpha(m_0^2+\lambda_0)}{2}\left\{(24(\|f\|_{L^\infty}+\|f\|_{L^\infty}^2)+\|f\|_{L^6}^2)E(q)+\frac{1}{2}\|f\|_{L^6}^4\right\}.
\end{align*}
Finally, we obtain
\begin{align*}
\dt G_1(q)=\dt E(q)+\dt I_\alpha(q)\leq -\alpha L_1(q)+\left(1+\alpha\frac{m_0^2+\lambda_0}{2}\right)\left\{(24\|f\|_{L^\infty}+\|f\|_{L^6}^2)G_1(q)+\|f\|_{L^6}^4\right\}.
\end{align*}
Applying Gronwall lemma we obtain $(\ref{ap_est_Energy}).$\\ To prove $(\ref{ap_est_G_2})$, let us compute
\begin{align*}
\dt G_2(v)&=(\dt v,\alpha\Delta_0\dt v -(3v^2f+3vf^2+f^3))-\frac{\alpha}{2}(\Delta_0 v, \Delta_0v-(v+f)^3))+\frac{\alpha}{2}\|\dt v\|_1^2\\
&\underbrace{+\frac{\alpha^2}{4}\dt\|u\|_2^2-\frac{\alpha}{2}(\Delta_0 u,\alpha\Delta_0\dt u)}_{=0}\\
&=(\dt v,-(3v^2f+3vf^2+f^3)-\frac{\alpha}{2}(\Delta_0 v, \Delta_0 v-(v+f)^3)-\frac{\alpha}{2}\|\dt v\|_1^2\\
&= I + II+III.
\end{align*}
By the first part of the proof, we have
\begin{align*}
I &\leq (24(\|f\|_{L^\infty}+\|f\|_{L^\infty}^2)+\|f\|_{L^6}^2)E(q)+\frac{1}{2}\|f\|_{L^6}^4,
\end{align*}
combining that with $(\ref{ineq_G_2>half_E}),$ we get
\begin{align*}
I\leq 2(24(\|f\|_{L^\infty}+\|f\|_{L^\infty}^2)+\|f\|_{L^6}^2)G_2(q)+\frac{1}{2}\|f\|_{L^6}^4=:R(f)G_2(q)++\frac{1}{2}\|f\|_{L^6}^4
\end{align*}
Now for any $\epsilon>0$, we have
\begin{align*}
II &\leq-\frac{\alpha}{2}\|v\|_2^2+\frac{\alpha}{2}\|v\|_2\|v+f\|_{L^6}^3\\
&\leq -\frac{\alpha}{2}\|v\|_2^2+\alpha\frac{\epsilon}{2}\|v\|_2^2+\frac{\alpha}{8\epsilon}\|v+f\|_{L^6}^6.
\end{align*}
Combining all this we obtain, for any $\epsilon>0,$ 
\begin{align*}
\dt G_2(q_t)+\frac{\alpha}{2}\left(\|\dt v\|_1^2+(1-\epsilon)\|v\|_2^2\right)&\leq G_2(q_t)R(f)+\frac{\alpha}{8\epsilon}\|v+f\|_{L^6}^6+\frac{1}{2}\|f\|_{L^6}^4,
\end{align*}
it remains to apply Gronwall lemma to arrive at the claim.
\end{proof} 
The following result will be used in the proof of Proposition $\ref{chap5:cond_kS}.$ 
\begin{prop}
For any $T>0$, any $\epsilon>0,$ we have the a-priori estimate 
\begin{align}\label{chap5_est_G1_moyenn}
E(q_T)+\alpha\int_0^T\|\dt v\|_{1}^2ds\leq e^{\frac{(9T+\epsilon)t}{\epsilon}+\frac{\epsilon}{T}\int_0^T\|f\|_{L^\infty}^2ds}\left(E(q_0)+\frac{1}{2}\int_0^T\left[\|f\|_{L^6}^6+\frac{T}{\epsilon}\|f\|_{L^4}^4\right]ds\right).
\end{align}
\end{prop}
\begin{proof}
We have that
\begin{align*}
\dt E(q) &=-(\dt v,3v^2f+3vf^2+f^3)+\alpha(\dt v,\Delta_0\dt v)\\
&=-3(f\dt v,vf+v^2)-(\dt v,f^3)-\alpha\|\dt v\|_1^2.
\end{align*}
Then
\begin{align*}
\dt E(q_t)+\alpha\|\dt v\|_1^2 &\leq \frac{3}{2}(\frac{\epsilon}{T}\|f\dt v\|^2+\frac{T}{\epsilon}\|vf+v^2\|^2)+\frac{1}{2}\|\dt v\|^2+\frac{1}{2}\|f\|_{L^6}^6.
\end{align*}
Notice that
\begin{equation*}
\frac{T}{\epsilon}\|vf+v^2\|^2)\leq \frac{2T}{\epsilon}(\|vf\|^2+\|v\|_{L^4}^4)\leq \frac{T}{\epsilon}(3\|v\|_{L^4}^4+\|f\|_{L^4}^4).
\end{equation*}
And then
\begin{align*}
\dt E(q_t)+\alpha\|\dt v\|_1^2 &\leq \frac{3\epsilon}{T}\|f\|_{L^\infty}^2E(q)+\frac{6T}{\epsilon}E(q)+E(q)+\frac{1}{2}\|f\|_{L^6}^6+\frac{T}{2\epsilon}\|f\|_{L^4}^4.
\end{align*}
One can change $\epsilon$ into $2\epsilon/3$ and use the Gronwall inequality to find, for all $t\geq 0,$
\begin{equation*}
E(q_t)+\alpha\int_0^t\|\dt v\|_{1}^2ds\leq e^{\frac{(9T+\epsilon)t}{\epsilon}+\frac{\epsilon}{T}\int_0^t\|f\|_{L^\infty}^2ds}\left(E(q_0)+\frac{1}{2}\int_0^t\left[\|f\|_{L^6}^6+\frac{T}{\epsilon}\|f\|_{L^4}^4\right]ds\right).
\end{equation*}
Then take $t=T$ to get the result.
\end{proof}
\subsection{Existence and uniqueness for the nonlinear equation $(\ref{nonlinear_flow})$}

Setting $w=\dt v$, we rewrite $(\ref{nonlinear_flow})$ into
\begin{equation}\label{linear_flow_mat}
\dt\underbrace{\left(\begin{array}{l r c}
v \\ w
\end{array}\right)}_{y}=\underbrace{\left(\begin{array}{l r c}
&0 &1\\
&\Delta_0 &\alpha\Delta_0
\end{array}\right)}_{A}\left(\begin{array}{l r c}
v \\ w
\end{array}\right)-\underbrace{\left(\begin{array}{l r c}
0 \\ (v+f)^3
\end{array}\right)}_{B(v+f)}.
\end{equation}
We denote by $S(t)$ the semi-group generated by $A$ thanks to the Hille-Yosida theorem.
\begin{prop}\label{ExistUniqnonlin} For any $q_0=[v_0,\dt{v}_0]\in\mathcal{H}^{1,0},$ for any $T>0,$ there is a unique $q(t,q_0)\in Z_1^T$ satisfying $(\ref{nonlinear_flow})$ with the the condition $q(0,q_0)=q_0.$ Moreover, the map $q_0\mapsto q(.,q_0)$ is continuous with respect to the underlying norms.
\end{prop}
\begin{proof}
 For a given $q_0\in \mathcal{H}^{1,0}$, for $T>0,$ we set the map $\psi:Z_1^T\to Z_1^T$:
\begin{equation}\label{duhamel_1}
\psi q(t)=S(t)q_0-\int_0^tS(t-s)B(v+f)ds,
\end{equation}
Let $R>0$, consider the ball $B_R$ in $Z_1^T$ centred at $0$ and of radius $R$, we show by standard arguments the following estimates:
\begin{equation*}
\|\psi q\|_{Z_1^T}\leq\tau_{T,R} \|q_0\|_{Z_1^T},
\end{equation*}
\begin{equation*}
\|\psi q_1-\psi q_2\|_{Z_1^T}\leq \tau_{T,R} \|q_1-q_2\|_{Z_1^T},
\end{equation*}
where $\tau_{T,R}$ decreases to $0$ with $T.$ Thus for arbitrary $R$, the time $T=T(R)$ can be choosen so that $\psi$ be a contraction as map from $B_R$ to $B_R,$ we have then a local in time existence
that we can globalize by iteration using the estimate $(\ref{ap_est_Energy})$.\\
For given two solutions $v_1,v_2$ to $(\ref{nonlinear_flow})$, set $w=v_1-v_2$. Then $w$ satisfies the equation
\begin{equation*}
\ds w-\Delta_0w=\alpha\Delta_0\dt w-w\{(v_1+f)^2+(v_2+f)^2+(v_1+f)(v_2+f)\},
\end{equation*}
it is not difficult to derive the following
\begin{equation}\label{nonlin_flow_uniq}
N_1(w,\dt{w})\lleq_{v_1,v_2,f} N_1(w(0),\dt{w}(0)).
\end{equation}
In fact $(\ref{nonlin_flow_uniq})$ establishes, at the same time, the continuity (in space) for the solution.
\end{proof}

\section{The stochastic linear KG equation, exponential control}
In this section, we present a treatment of the following equation:
\begin{align*}
\dt[z,\dt z] &=[\dt z,\Delta_0z+\alpha\Delta_0\dt z]+\sqrt{\alpha}[0,\eta]\nonumber\\
&=A[z,\dt z]+\sqrt{\alpha}\dt \hat{\zeta},
\end{align*}
with null initial condition, supplemented with the Dirichlet condition in the case where the equation is considered on a bounded domain. Then the solution is given by the following stochastic convolution (see Section $\ref{intro_section_convol_stoch_et_Ito_lemma}$ for a definition and some properties):
\begin{equation*}
[z,\dt z](t)=\sqrt{\alpha}\int_0^tS(t-s)d\hat{\zeta}(s).
\end{equation*}
In view of the discussion in Section $\ref{intro_section_convol_stoch_et_Ito_lemma}$, $[z,\dt z]$ belongs to $\mathcal{H}^{1,0}$ (resp. $\mathcal{H}^{2,1}$) if $A_0$ (resp. $A_1$) is finite. In what follows we suppose that $A_1$ is finite. An exponential control for $[z,\dt z]$ is given for any $t$ by the Fernique theorem, here we prove an exponential control on the time-averaged norm. Such a control will be used to prove Proposition $\ref{chap5:cond_kS}.$
\begin{prop}\label{chap_5: exp_cont_lin}
Let $0<\epsilon\leq \kappa /(2A_1e)$, we have
\begin{equation}\label{chap5_est_exp_z}
\E e^{\frac{\epsilon}{t}\int_0^t\|[z,\dt z]\|_{2,1}^2ds}\leq 3,
\end{equation}
where $\kappa=\min(1,m_0^2+\lambda_0).$
\end{prop}
\begin{proof}
\textbf{Step $1$: The finite-dimensional approximating equation and estimation of the moments.}
Let $P_N$ be the projection on the finite-dimensional space $E^N:=\text{span}\{e_0,...,e_N\}$. Set $z^N=P_Nz,$ $A^N=P_NA$, $\hat{\zeta}^N=P_N\hat{\zeta}$, $\Delta^N=P_N\Delta$ and $A_{m,N}=\sum_{j=0}^N(m_0^2+\lambda_j)^ma_j^2$. Then we have
\begin{equation*}
\dt[z^N,\dt z^N]=A^N[z^N,\dt z^N]+\sqrt{\alpha}\dt\hat{\zeta}^N.
\end{equation*}
It is a matter of direct verification that the norm $f(z^N,\dt z^N):=\|[z^N,\dt z^N]\|_{2,1}^2$ is still preserved by the approximating equation in which we take $\alpha=0.$ The fonction $f$ belongs to $C^2(E^N\times E^N,\R)$, then we can apply the finite-dimensional Itô formula:
\begin{equation*}
df(z^N,\dt z^N)=\alpha\left(\frac{A_{1,N}}{2}-\|[z^N,\dt z^N]\|_{2,2}^2\right)dt+\sqrt{\alpha}\sum_{m=0}^Na_m((-\Delta_0^N)^{1/2}\dt z^N,(-\Delta_0^N)^{1/2}e_m)d\beta_m.
\end{equation*}
Now let $p>1$, we have that $f^p$ still belongs to $C^2(E^N\times E^N,\R)$, the Itô formula gives that
\begin{align*}
df^p(z^N,\dt z^N) &=pf^{p-1}df+\frac{\alpha p(p-1)}{2}\sum_{j=0}^Na_m^2f^{p-2}((-\Delta_0^N)^{1/2}\dt z^N,(-\Delta_0^N)^{1/2}e_m)^2dt\\
&=:(1)+(2)
\end{align*}
\begin{align*}
(1)=\alpha p\|[z^N,\dt z^N]\|_{2,1}^{2(p-1)}\left(\frac{A_{1,N}}{2}-\|[z^N,\dt z^N]\|_{2,2}^2\right)dt+\theta(t),
\end{align*}
where $\theta(t)$ is the stochastic integrand and verifies $\E\int_0^t\theta(s)=0.$ We see that
\begin{align*}
(1) \leq \alpha p\|[z^N,\dt z^N]\|_{2,1}^{2(p-1)}\frac{A_{1,N}}{2}-\alpha p\kappa\|[z^N,\dt z^N]\|_{2,1}^{2p}+\theta(t).
\end{align*}
 On the other hand,
\begin{align*}
(2)\leq \frac{\alpha p(p-1)}{2}\|[z^N,\dt z^N]\|_{2,1}^{2(p-1)}A_{1,N}.
\end{align*}
Then, with use of the Young inequality, 
\begin{align*}
df^p(z^N,\dt z^N)-\theta(t)&\leq -\alpha p\kappa\|[z^N,\dt z^N]\|_{2,1}^{2p}dt+\frac{\alpha p^2}{2}\|[z^N,\dt z^N]\|_{2,1}^{2(p-1)}A_{1,N}dt \\
&\leq -\alpha p\kappa\|[z^N,\dt z^N]\|_{2,1}^{2p}dt+\frac{\alpha p\kappa}{2}\|[z^N,\dt z^N]\|_{2,1}^{2p}dt+\frac{\alpha A_1^p p^{p+1}}{2\kappa^{p-1}}dt.
\end{align*}
After integrating in $t$, taking the expectation, and using the Gronwall lemma, we get
\begin{equation*}
\E\|[z^N,\dt z^N]\|_{2,1}^{2p}\leq \frac{A_1^p p^{p}}{\kappa^{p}}.
\end{equation*} 
\\ \textbf{Step $2$: Passage to the limit $N\to\infty$.}
Using Fatou's lemma, we get the estimations for $[z,\dt z]$
\begin{equation}\label{chap5_estim_power,z}
\E\|[z,\dt z]\|_{2,1}^{2p}\leq \frac{A_1^p p^{p}}{\kappa^{p}}.
\end{equation}
\\ \textbf{Step $3$: Exponential control.}
Integrating in $t$, we find
\begin{equation*}
\E\left(\frac{1}{t}\int_0^t\|[z,\dt z]\|_{2,1}^{2p}ds\right)\leq \frac{A_1^p p^{p}}{\kappa^{p}}.
\end{equation*}
Thanks to Jensen's inequality, we infer
\begin{equation*}
\E\left(\frac{1}{t}\int_0^t\|[z,\dt z]\|_{2,1}^{2}ds\right)^p\leq \frac{A_1^p p^{p}}{\kappa^{p}}.
\end{equation*}
Now, let $0<\epsilon\leq \kappa /(2A_1e),$ then we have
\begin{equation*}
\E\frac{\left(\frac{\epsilon}{t}\int_0^t\|[z,\dt z]\|_{2,1}^{2}ds\right)^p}{p!}\leq \frac{p^{p}}{2^{p}e^pp!}.
\end{equation*}
We recall that for any integer $p>0$, we have that $p!\geq\left(\frac{p}{e}\right)^p,$ then we arrive at the claimed result.
\end{proof}
\begin{rmq}
One could use directly the infinite-dimensional Itô formula, the requirements of the latter are satisfied using the Fernique theorem.
\end{rmq}
\section{Global dynamics of the damped-perturbed KG and well structuredness}
In what follow we use the following notations:
\begin{align*}
y_t &=[u,\dt u],\ \ \ g=[\dt{u},\ \Delta_0u-u^3+\alpha\Delta_0\dt u]=:[g_1,g_2],\\
\hat{\zeta}(t)&=\sum_{m=0}^\infty a_m\hat{e}_m  \beta_m(t), \ \ \ \hat{e}_m=[0,e_m].
\end{align*}
In the rest of the paper we suppose that the quantities $A_i=\sum_{j=0}^\infty a_m^2(m_0^2+\lambda_m)^i$ are finite for $i=0,1$.

\begin{prop}\label{KG_stoch_WP}
Let $m\in\{1,2\}$. The equation $(\ref{ddKG})$ is well structured on \\ $(\mathcal{H}^{m,m-2},\mathcal{H}^{m,m-1},\mathcal{H}^{m,m})$ in the sense of Definition $\ref{intro_WS}$.
\end{prop}

\begin{proof}
First, we prove the stochastic global wellposedness on $\mathcal{H}^{1,0}$ in the following steps:\\
\textbf{Step $1$: Splitting the problem.} In order to solve the initial value problem of the equation $(\ref{ddKG})$, we split the latter into the following two equations:
\begin{equation}\label{linear_eq}
\ds z_\alpha-\Delta_0 z_\alpha=\alpha\Delta_0\dt z_\alpha+\sqrt{\alpha}\eta,
\end{equation}
\begin{equation}\label{nonlinear_eq}
\ds v-\Delta_0 v+(v+z_\alpha)^3=\alpha\Delta_0\dt v.
\end{equation}
Under the initial conditions $[z_\alpha,\dt z_\alpha]|_{t=0}=[0,0],[v,\dt v]|_{t=0}=[u,\dt u]|_{t=0}$, we have that $u=v+z_\alpha$ solves $(\ref{ddKG})\footnote{Notice that this kind of decompositon appears in litterature of stochastic PDEs (see e.g. \cite{KS12}) and of dispersive PDEs (see e.g. \cite{burktzvt_probwav} for the context of cubic wave equation).}.$ Therefore it suffices to solve each of these two equations.

 \textbf{Step $2$: The linear stochastic problem.}
The linear equation $(\ref{linear_eq})$, supplemented by the initial data $z_\alpha|_{t=0}=\dt z_\alpha|_{t=0}=0,$ is solved by the following stochastic convolution
\begin{equation*}
[z_\alpha,\dt z_\alpha](t)=\sqrt{\alpha}\int_0^tS(t-s)d\hat{\zeta}(s).
\end{equation*}
The solution $[z_\alpha,\dt z_\alpha]$ is almost surely in $Z_1^\infty$ when $A_0$ is finite (see Subsection $\ref{intro_section_convol_stoch_et_Ito_lemma}$ for properties of the stochastic convolution). 

\textbf{Step $3:$ The nonlinear deterministic problem.}
The initial value problem of the nonlinear equation $(\ref{nonlinear_eq})$ is solved by a deterministic way. Suppose $A_0$ finite. Fix $\omega$ for which $z_\alpha \in Z_1^\infty,$ we can then take $f=z_\alpha$ in $(\ref{nonlinear_flow})$ and the problem is solved in view of Proposition $\ref{ExistUniqnonlin}.$

\textbf{Step $4:$ Progressive measurability and continuity.}
By the definition of $z_\alpha$ and $v$, we have that the solution  $u=v+z_\alpha$ is $\sigma(u_0,\dt{u}_0,\mathcal{F}_t)-$adapted and $u$ is continuous in time (with values in $H^1$). Then using the Proposition $1.13$ of \cite{karatshre}, we get the progressive measurability for $u$.\\
The continuity (w.r.t. initial data) property follows that established for the "nonlinear solution" $v$, since the "linear solution" $z$ does not depend on the initial data.\\
Now, we prove that the solution $y_t$ satisfies the assumptions $(\ref{intro_hyp_KS})$ on the Gelfand triples \\$(\mathcal{H}^{m,m-2},\mathcal{H}^{m,m-1},\mathcal{H}^{m,m})$ for $m=1,2:$
\begin{enumerate}
\item $y_t$ is a Itô process in $\mathcal{H}^{m,m-2}$ since the process $g(y):=[\dt u,\Delta_0u-u^3+\alpha\Delta_0\dt u]$ is $\mathcal{F}_t$-adapted, and we infer from Proposition $\ref{prop_estimat_genral}$ the following
\begin{equation*}
\P\left(\int_0^t\|g(y_s)\|_{\mathcal{H}^{m,m-2}}^2ds <\infty\ \ \ for \ all\ \  t\geq 0\right)=1.
\end{equation*}
\item The quantities $A_0,\ A_1$ are finite by assumption;
\item Again, we use the estimates of Proposition $\ref{prop_estimat_genral}$ to see that
\begin{equation*}
\P\left(\int_0^t\|y_s\|^2_{\mathcal{H}^{m,m}}ds<\infty\ \ \ for \ all\ \  t\geq 0\right)=1.
\end{equation*}
\end{enumerate}
The proof is complete.
\end{proof}

In view of the wellposedness established above, we are able to define the flow map of $(\ref{ddKG})$
\begin{equation*}
\phi_r^\alpha w=y^\alpha(r,w),
\end{equation*}
the transition function 
\begin{equation*}
P_t^\alpha(w,E)=\P(\phi_t^\alpha w\in E),
\end{equation*}
and the Markov semigroups, with use of the Feller property induced by the continuity of the flow,
\begin{align*}
\p f(w) &=\int f(v)P_t^\alpha(w,dv), \ \ \ C_b(\mathcal{H}^{1,0})\to C_b(\mathcal{H}^{1,0});\\
\q \nu(E) &=\int P_t^\alpha(v,E)\nu(dv),\ \ \  \mathfrak{p}(\mathcal{H}^{1,0})\to\mathfrak{p}(\mathcal{H}^{1,0}),
\end{align*}
where $\mathfrak{p}(H)$ is the set of probability measures on $H$. The functions $\p$ and $\q$ verify the duality relation
\begin{equation*}
(\p f,\nu)=(f,\q v).
\end{equation*}

\section{Stationary measures for the damped-perturbed KG}\label{sect_prv}
We suppose that $A_0$ is finite and recall the notation $\gamma_0=\frac{2\kappa^2}{2+m_0^2+\lambda_0}$.
\begin{thm}\label{Thm_alpha}
For any $\alpha\in (0,1),$ the problem $(\ref{ddKG})$ admits an invariant measure $\mu_\alpha$ defined concentrated on $\mathcal{H}^{2,1}$. The invariant measures $\mu_\alpha$ of $(\ref{ddKG})$ satisfy the following properties
\begin{enumerate}
\item For any $\alpha\in (0,1)$
 \begin{equation}\label{estim_unif_L1}
\int_{\mathcal{H}^{1,0}}L_1(y)\mu_\alpha(dy)=\frac{A_0}{2}.
\end{equation}
\item For any $p\geq 1$, we have
\begin{align}
\int_{\mathcal{H}^{1,0}} G^p_1(y)\mu_\alpha(dy) &\leq \left(\frac{2pA_0}{\gamma_0}\right)^p.\label{estim_unif_Gp}
\end{align}
\item There is $C$ independent of $\alpha$ such that
\begin{equation}\label{estim_unif_H2H1}
\int_{\mathcal{H}^{1,0}}\|y\|^2_{\mathcal{H}^{2,1}}\mu_\alpha(dy)\leq C.
\end{equation}
\end{enumerate}
\end{thm}

\subsection{Step 1: Statistical controls}
We would like to apply Itô's formula ((Theorem $A.7.5$ and Corollary $A.7.6$ of \cite{KS12})) to the functionals $G_1$ w.r.t. the triple $(\mathcal{H}^{1,-1},\mathcal{H}^{1,0},\mathcal{H}^{1,1})$ and $G_2$ w.r.t. $(\mathcal{H}^{2,0},\mathcal{H}^{2,1},\mathcal{H}^{2,2})$.
The polynomial structure of these functionals allows to fill the conditions of Theorem $A.7.5$ of \cite{KS12}. Here, we wish to apply the Itô formula with a deterministic time, then we have to verify the condition of the Corollary $A.7.6$ of the same book, namely the finiteness of the quadratic variation of each of these functionals.
\begin{prop}\label{chap5:cond_kS}
Suppose that $\E E^{\iota}(y_0) <\infty$ with $\iota>1.$ Then the quantities $G_1(y_t)$ and $G_2(y_t)$ have finite quadratic variations on any finite interval.
\end{prop} 
\begin{proof}
We have for $i=1,2,$
\begin{align*}
\sum_{m\geq 0}a_m^2 &\E\int_0^t|\nabla_y G_i(y;\hat{e}_m)|^2ds \lleq \sum_{m\geq 0}a_m^2\E\int_0^t|(\dt u+\frac{\alpha}{2}(-\Delta_0)^{i-1}u;e_m)|^2ds\\
&\lleq \sum_{m\geq 0}a_m^2\E\int_0^t\left\{|(\dt v+\frac{\alpha}{2}(-\Delta_0)^{i-1}v;e_m)|^2+(\dt z+\frac{\alpha}{2}(-\Delta_0)^{i-1}z;e_m)|^2\right\}ds\\
&\lleq_{A_0,A_1}\E\int_0^tE(q_s)ds+\E\int_0^t(\|[z,\dt z]\|_{i-1,0}^2)ds.
\end{align*}
One see, with use of estimates $(\ref{chap5_estim_power,z})$, that
\begin{equation*}
\E\int_0^t(\|[z,\dt z]\|_{i-1,0}^2)ds<\infty \ \ \ for\ \ any \ \ t\geq 0.
\end{equation*}
Now we use estimate $(\ref{chap5_est_G1_moyenn})$, that, for any $\epsilon>0,$ 
\begin{equation*}
\E\int_0^tE(q_s)ds\leq \int_0^t e^{\frac{9s^2+\epsilon s}{\epsilon}}\E\left[e^{\frac{\epsilon}{s}\int_0^s\|f\|_{L^\infty}^2dr}\left(E(q_0)+\frac{1}{2}\int_0^s\left[\|f\|_{L^6}^6+\frac{T}{\epsilon}\|f\|_{L^4}^4\right]dr\right)\right]ds.
\end{equation*}
By the Young inequality, we have for any $\epsilon>0,$
\begin{equation*}
\E\int_0^tE(q_s)ds\lleq \int_0^t e^{\frac{9s^2+\epsilon s}{\epsilon}}\E\left[e^{\frac{1^+\epsilon}{(1^+-1)s}\int_0^s\|f\|_{L^\infty}^2dr}+\underbrace{\left(G_1(q_0)+\frac{1}{2}\int_0^s\left[\|f\|_{L^6}^6+\frac{T}{\epsilon}\|f\|_{L^4}^4\right]dr\right)^{1^+}}_{R(s)}\right]ds.
\end{equation*}
One uses the estimate $(\ref{chap5_estim_power,z})$ and the Jensen inequality to bound
$\E R(s)$ by in $C(1+s^{1^+}).$ And, for  small enough $\epsilon>0$ (here $\epsilon$ depends indeed on the infinitesimal parameter enterring the definition of $1^+$), the estimate $(\ref{chap5_est_exp_z})$ and the embedding $H^2\subset L^\infty$ allow to get the bound
\begin{equation*}
\E e^{\frac{1^+\epsilon}{(1^+-1)s}\int_0^s\|f\|_{L^\infty}^2dr}\leq 3.
\end{equation*} 
Then we get
\begin{equation*}
\E\int_0^tE(q_s)ds\lleq \int_0^te^{\frac{9s^2+\epsilon s}{\epsilon}}(1+s^{1^+})ds <\infty \ \ \ for \ \ all\ \ t\geq0.
\end{equation*}
The proof is finished.\end{proof}
Taking the inequality $(\ref{chap5_est_G1_moyenn})$ to the power $p>1$ and repeating the above argument, we show that $G_1^p$ satisfies $(\ref{intro_KS_Ito_F3})$ as soon as $\E E^{p^+}(y_0)$ is finite. 

\begin{prop}\label{ito_sur_functionals}
Let $\alpha \in (0,1).$ Suppose  $\E G_1^p(y_0)<+\infty$ for any $p>1$. Then the solution $y_t$ of $(\ref{ddKG})$ starting at $y_0$ satisfies the following
\begin{align}
\E G_1(y_t) &+\alpha\int_0^t\E L_1(y_s)ds = \E G_1(y_0)+\frac{\alpha}{2}A_0t,\label{Ito_est11}\\
\E G_1^p(y_t) &\leq e^{-\alpha\gamma_0\frac{p}{2}t}\E G_1^p(y_0)+2\left(\frac{2pA_0}{\gamma_0}\right)^p.\label{Ito_est12}
\end{align}
\end{prop}

\begin{proof}
For a functional $F(y)$, we denote by $\nabla_uF$ and $\nabla_vF$ the derivatives w.r.t. the first and the second variable respectively. Let us compute
\begin{align*}
\nabla_yG_1(y,g) &=\nabla_uG_1(y,g_1)+\nabla_{v}G_1(y,g_2)=\nabla_uE(y,g_1)+\nabla_{v}E(y,g_2)+I+II\\
&=-\alpha\|\dt u\|_1^2+I+II,\\
I &=\frac{\alpha}{2}\|\dt u\|^2,\\
II &=\frac{\alpha}{2}(u,g_2)+\frac{\alpha^2}{4}\dt\|u\|_1^2= -\frac{\alpha}{2}(\|u\|_1^2+\|u\|_{L^4}^4),
\end{align*}
thus
\begin{align*}
\nabla_yG_1(y,g)=-\frac{\alpha}{2}((m_0^2+\lambda_0)\|u\|_1^2+2\|\dt u\|_1^2-(m_0^2+\lambda_0)\|\dt u\|^2+(m_0^2+\lambda_0)\|u\|_{L^4}^4)=-\alpha L_1(y).
\end{align*}
On the other hand
\begin{equation*}
\nabla_y^2G_1(y,\hat{e}_m)=\nabla_u^2G_1(y,0)+\nabla_{v}^2G_1(y,e_m)=(e_m,e_m)=1.
\end{equation*}
Then, by Itô formula (see \cite{KS12}, Theorem $A.7.5$ and Corollary $A.7.6$), we have
\begin{align*}
dG_1(y_t)=-\alpha L_1(y)dt+\frac{\alpha}{2}A_0dt+\Theta_1(t),
\end{align*}
where
\begin{equation*}
\Theta_1(t)=\sum_{m=0}^\infty a_m\nabla_yG_1(y,\hat{e}_m)d\beta_m(t)=\sum_{m=0}^\infty a_m(\dt u+\frac{\alpha}{2}u,e_m)d\beta_m(t).
\end{equation*}
Remark that, by an Itô integral property, $\E\int_0^t\Theta_1(s)=0,$ then we arrive at $(\ref{Ito_est11}).$ Now, let $p> 1$, we have, by Itô formula,
\begin{align*}
dG_1^p(y) &=pG_1^{p-1}dG_1(y)+\frac{\alpha p(p-1)}{2}\sum_{m\geq 0}a_m^2G_1^{p-2}(y)(\nabla_yG_1(y,e_m))^2dt\\
&=pG_1^{p-1}dG_1(y)+\frac{\alpha p(p-1)}{2}\sum_{m\geq 0}a_m^2G_1^{p-2}(y)(\dt u,e_m)^2dt. 
\end{align*}
whence it follows that
\begin{align*}
\E G_1^{p}(y_t)+\int_0^t\E f_\alpha(s)ds\leq \E G_1^p(y_0),
\end{align*}
where
\begin{align*}
f_\alpha(t) &= p\frac{\alpha}{2}G_1^{p-1}(2L_1(y)-A_0)- \frac{\alpha p(p-1)}{2}G_1^{p-2}(y)\sum_{m\geq 0}a_m^2(\dt u,e_m)^2.
\end{align*}
We have, with the use of the inequalities $(\ref{E<F_1<E})$, $(\ref{G_1pgE1})$ and the inequality $2p^2-2p\leq 2p^2-p/2$ for $p\geq 0,$
\begin{align*}
f_\alpha(t)&\geq p\frac{\alpha}{2}G_1^{p-1}(y)\left(2L_1(y)-A_0\right)-2p(p-1)\alpha A_0G_1^{p-1}(y)\\
&\geq p\alpha G_1^{p-1}(y)L_1(y)-2\alpha p^2A_0G_1^{p-1}(y)\\
&\geq \alpha\gamma_0pG_1^{p}(y)-2\alpha p^2A_0G_1^{p-1}(y)\\
&\geq \frac{\alpha\gamma_0p}{2}G_1^p(y)-\alpha\frac{2^{p}p^{p+1}}{\gamma_0^{p-1}}A_0^p.
\end{align*}
Finally
\begin{equation*}
\E G_1^p(y_t)+\frac{\alpha\gamma_0p}{2}\int_0^t\E G_1^p(y_s)ds\leq \E G_1^p(y_0)+\alpha \frac{2^{p}p^{p+1}}{\gamma_0^{p-1}}A_0^pt.
\end{equation*}
Thus we get $(\ref{Ito_est12})$ after applying the Gronwall lemma.
\end{proof}
\begin{prop}
Let $\alpha \in (0,1).$ Suppose  $\E G_2^{1^+}(y_0)<+\infty.$ Then the solution $y_t$ of $(\ref{ddKG})$ starting at $y_0$ satisfies the following
\begin{align}
\E G_2(y_t) &+\alpha\int_0^t\E L_2(y_s)ds \leq \E G_2(y_0)+ c_1\alpha\left(t+\int_0^t\E\|u\|_{L^6}^6ds\right),\label{Ito_est21}\\
\E G_2(y_t) &\leq e^{-c_2t}\E G_2(y_0)+c_3,\label{Ito_est22}
\end{align}
where $c_i, \ i=1,2,3$ are universal positive constants.
\end{prop}
We remark in $(\ref{Ito_est21})$ we need to control $\E\|u\|_{L^6}^6.$ But one can see that this control holds true if one combines the embedding $H^1\subset L^6$ and the estimate $(\ref{Ito_est12}).$
\begin{proof}
Set $J_\alpha(y)=-\frac{\alpha}{2}\int \dt u\Delta_0 u+\frac{\alpha^2}{4}\|u\|_2^2$ so that $G_2=E+J_\alpha.$ In order to apply the Itô formula used in the previous estimations, let us compute
\begin{align*}
\nabla_y G_2(y,dy) &=\nabla_yE+\nabla_u J_\alpha(y,g_1)+\nabla_{v}J_\alpha(y,g_2)+\Theta_2=:\nabla_yE+I+II+\Theta_2,
\end{align*}
where 
\begin{align*}
\Theta_2(t)=\sqrt{\alpha}\sum_{m\geq 0}a_m(\dt u-\frac{\alpha}{2}\Delta_0 u,e_m)d\beta_m(t).
\end{align*}
We see without any difficulty that
\begin{equation*}
\nabla_yE(y,g) =-\alpha\|\dt u\|_1^2.
\end{equation*}
On the other hand, we have
\begin{align*}
I &=\frac{\alpha}{2}\|\dt u\|_1^2+\frac{\alpha^2}{4}\dt\|u\|_2^2.\\
II &=-\frac{\alpha}{2}(\Delta_0 u,\Delta_0u-u^3+\alpha\Delta_0\dt u)\\
&=-\frac{\alpha}{2}\|u\|_2^2+\frac{\alpha}{2}(\Delta u,u^3)-\frac{\alpha^2}{4}\dt\|u\|_2^2\\
&\leq -\frac{\alpha}{2}(1-\epsilon)\|u\|_2^2+\frac{\alpha}{8\epsilon}\|u\|_{L^6}^6-\frac{\alpha^2}{4}\dt\|u\|_2^2.
\end{align*}
Summarizing all this, we have
\begin{align*}
\nabla_yG_2(y,dy)\leq-\frac{\alpha}{2}(\|\dt u\|_1^2+(1-\epsilon)\|u\|_2^2)+\frac{\alpha}{8\epsilon}\|u\|_{L^6}^6+\Theta_2,
\end{align*}
where $\Theta_2$ satisfies $\E\int_0^t\Theta_2(s)=0$ for any $t>0.$
Now
\begin{align*}
\nabla_y^2G_2(y,\hat{e}_m)=1,
\end{align*}
then for any $\epsilon \in (0,1),$ we have
\begin{align*}
\E G_2(y_t)+\frac{\alpha}{2}\int_0^t\E(\|\dt u\|_1^2+(1-\epsilon)\|u\|_2^2)ds &\leq \E G_2(y_0)\\
&+\frac{\alpha}{8\epsilon}\int_0^t\E\|u\|_{L^6}^6ds+\frac{\alpha}{2} A_0t,
\end{align*}
that is $(\ref{Ito_est21}).$ To prove $(\ref{Ito_est22}),$ we remark that 
\begin{align*}
L_2\ggeq G_2-E.
\end{align*}
Injecting that into $(\ref{Ito_est21})$, we find
\begin{equation*}
\E G_2(y)+c_2\alpha\int_0^t\E G_2(y_s)ds\leq \E G_2(y_0)+\alpha\int_0^t\E(E(u)+\|u\|_{L^6}^6)ds.
\end{equation*}
Then we use $(\ref{Ito_est12})$ and Gronwall lemma to conclude.
\end{proof}
From the estimate $(\ref{Ito_est21})$, we infer, for $y_0=0$ $a.e,$ that
\begin{equation}\label{est_H3}
\frac{1}{t}\int_0^t\E\|y_s\|^2_{\mathcal{H}^{2,1}}ds\leq C,
\end{equation}
where $C$ is independent of $t.$

\subsection{Step $2$: Stationary measures and their estimations}
\paragraph{\textbf{Existence of a stationary measure on $\mathcal{H}^{1,0}$ for any fixed $\alpha$.}}
Let $\delta_0$ be the Dirac measure on $\mathcal{H}^{1,0}$ concentrated at $0$. We define the time-averaged measures
\begin{equation}
\bar{\lambda}_T=\frac{1}{T}\int_0^T\q \delta_0 dt.\label{ito_recap}
\end{equation}
Let us fix $\alpha$. Now, we are going to establish (uniform) tighness of  the sequence $\{\bar{\lambda}_T,\ T>0\}$.
Let $R>0$ and $B_R$ be the ball of $\mathcal{H}^{2,1}$ of center $0$ and radius $R$, $B_R$ is compact in $\mathcal{H}^{1,0}$ and, thanks to $(\ref{est_H3})$, Chebychev inequality implies
\begin{equation*}
\bar{\lambda}_T(B_R)\geq 1-\frac{C}{R^2},
\end{equation*}
where $C$ is independent of $T$. Then $\{\bar{\lambda}_T,\ T\}$ is tight on $\mathcal{H}^{1,0}.$ By Prokhorov theorem, there is an accumulation point on $\mathcal{H}^{1,0}$, then the Bogoliubov-Krylov argument establishes that the latter is invariant for $(\ref{KG})$. We denote this stationary measure by $\mu_\alpha.$
\paragraph{Uniform (in $\alpha$) estimates for the measures $\mu_\alpha.$}
Now we prove the estimates $(\ref{estim_unif_L1}),$ $(\ref{estim_unif_Gp}),(\ref{estim_unif_H2H1}).$ Let $R$ be a positive number. Consider a $C^\infty-$function $\chi_R$ on $\R$ defined by
\begin{equation*}
\chi_R(x)=\left\{\begin{array}{l r c}
1, \ \ \text{if $x\leq R,$}\\
0, \ \ \text{if $x\geq R+1$}.
\end{array}
\right.
\end{equation*}
Let us prove $(\ref{estim_unif_Gp}),$ the proof of  $(\ref{estim_unif_H2H1})$  is similar and $(\ref{estim_unif_L1})$ follows the finiteness of $\E G_1(y)$ and $(\ref{Ito_est11})$. For any $p\geq 1,$ we have
\begin{equation}\label{trans_identity}
\int_{\mathcal{H}^{1,0}}G_1^p(y)\chi_R(\|y\|_{1,0})\mu_\alpha(dy)=\int_{\mathcal{H}^{1,0}}\E[G_1^p(y(t,v))\chi_R(\|y(t,v)\|_{1,0})]\mu_\alpha(dv).
\end{equation}
Passing to the limit $t\to\infty$ in $(\ref{trans_identity})$ with use of $(\ref{Ito_est12})$, we arrive at
$$\int_{\mathcal{H}^{1,0}}G_1^p(y)\chi_R(\|y\|_{1,0})\mu_\alpha(dy)\leq 3p^pA_0^p,$$
it remains to apply the Fatou lemma to finish. Now the control on the $\mathcal{H}^{2,1}-$ norm implies
\begin{align*}
\mu_\alpha(\mathcal{H}^{2,1})=1.
\end{align*}
\section{Invariant measure for KG and estimates}
\begin{thm}\label{Thm_sans_alpha}
There is an accumulation point $\mu$ of $\{\mu_\alpha\}$ as $\alpha\to 0$, in the weak topology of $\mathcal{H}^{1,0}$, satisfying the following properties:
\begin{enumerate}
\item $\mu$ is invariant under the flow of the KG equation $(\ref{KG})$ defined on $\mathcal{H}^{1,0};$
\item \begin{equation}
\mu (\mathcal{H}^{2,1})=1;\label{inviscid_concent}
\end{equation}
\item 
 \begin{align}
\int_{\mathcal{H}^{1,0}} \|y\|_{\mathcal{H}^{2,1}}^2\mu(dy) &<+\infty;\label{inviscid_normH2H1}
\end{align}
\item for any $p\geq 1$,
\begin{align}
\int_{\mathcal{H}^{1,0}} E^p(y)\mu(dy) \leq 2\left(\frac{2pA_0}{\gamma_0}\right)^p,\label{inviscid_Ep}
\end{align}
where $A_0=\sum_{m=0}^\infty a_m^2$ and $\gamma_0=\frac{2\min(1,\ m_0^2+\lambda_0)}{2+m_0^2+\lambda_0}.$
\end{enumerate}
\end{thm}

In what follows $B_R$ denotes the ball in $\mathcal{H}^{2,1}$ centred at zero and of radius $R$, unless otherwise specified. $\phi_t^\alpha$ and $\phi_t$ denote respectively the flows of $(\ref{ddKG})$ and $(\ref{KG})$ on $\mathcal{H}^{1,0}.$ The Markov semi-groups  associated to $(\ref{KG})$ are denoted by $
\mathfrak{P}_t$ and
$\mathfrak{P}^*_t$.

For $w\in \mathcal{H}^{1,0},$ $v(t,w)$ and $u(t,w)$ are the corresponding solutions to the nonlinear equation $(\ref{nonlinear_eq})$ and the KG equation $(\ref{KG})$. Set  $f=v-u,$ then $f|_{t=0}=0.$ We have
\begin{nem} Let $T,$ $R$ and $r$ be positive numbers, we have
\begin{equation}\label{Convergence}
\sup_{t\in[0,T]}\sup_{w\in B_R}\E\left[E(f,\dt{f})\Bbb 1_{\{\|z_\alpha\|_{L^\infty_tH^2}\leq \sqrt{\alpha}r\}}\right]=O_{T,R,r}(\alpha).
\end{equation}
\end{nem}
\begin{proof} 
Consider the following equations
\begin{align*}
\ds v-\Delta_0v+v^3 &=\alpha\Delta_0\dt v-3v^2z_\alpha-3vz_\alpha^2-z_\alpha^3,\\
\ds u-\Delta_0u+u^3 &=0.
\end{align*}
Taking the difference between the above two equations, we get
\begin{align*}
\ds f-\Delta_0 f+f^3 &=\alpha\Delta_0\dt v+3u^2v-3v^2u-3v^2z_\alpha-3vz_\alpha^2-z_\alpha^3\\
&=\alpha\Delta_0\dt v-3uvf-3v^2z_\alpha-3vz_\alpha^2-z_\alpha^3.
\end{align*}
Let $t\in [0,T]$ and $w\in B_R.$ Thanks to the preservation of $E$ by the solution of $(\ref{KG})$ and the results of Proposition $\ref{prop_estimat_genral}$,  we get
\begin{align*}
\dt E(f,\dt{f}) &=-3(f\dt f,uv)+\alpha(\dt f, \Delta_0\dt v)-(\dt f,3v^2z_\alpha+3vz_\alpha^2+z_\alpha^3)\\
&\lleq \|\dt f\|\|f\|_{L^4}\|uv\|_{L^4}+\alpha\|\dt f\|_1\|\dt v\|_1+\|\dt f\|O_{R,r}(\sqrt{\alpha})\\
&\lleq (\|\dt f\|^2+\|f\|_{L^4}^2)[\|u\|_{L^4}\|v\|_2+1]+O_{R,r}(\alpha)\\
&\lleq E(f,\dt{f})[\|u\|_{L^4}^2+\|v\|_2^2+1]+O_{R,r}(\alpha)\\
&\lleq_{T,R,r} (1+L_2(q))\ E(f,\dt{f})+\alpha.
\end{align*}
Now Proposition $\ref{prop_estimat_genral}$ ensures the boundedness of $\int_0^tL_2(q)ds$ where $q=[v,\dt{v}].$
We apply Gronwall lemma to get the claimed result.
\end{proof}
Since for $a.a.$ $\omega$ $z_\alpha$ converges to $0$ as $\alpha\to 0$, we have
\begin{cor}\label{coro}
For any $T,R,r>0$ and almost all $\omega\in\Omega$,
\begin{equation*}
\sup_{t\in[0,T]}\sup_{w\in B_R}\E\left[\|\phi_t^\alpha w-\phi_t w\|_{\mathcal{H}^{1,0}}^21_{\{\|z_\alpha\|_{L^\infty_tH^2}\leq \sqrt{\alpha}r\}}\right]=O_{T,R,r}(\alpha).
\end{equation*}
\end{cor}
\begin{proof}[Proof of Theorem $\ref{Thm_sans_alpha}$]
The family  $\{\mu_\alpha\}$ is tight on $\mathcal{H}^{1,0}$ $w.r.t.$ $\alpha$ by $(\ref{estim_unif_H2H1}),$ then passing to a subsequence, we have a limiting measure $\mu.$ In what follow the subscript $k$ is related to $\alpha_k$, the $k$th term of the above subsequence.
\begin{enumerate}
\item \textbf{Estimates for the inviscid limit.}
Let $\chi_R$ be a bump function on $[0,R]$ for $R>0$. By $(\ref{estim_unif_L1})$, we have
\begin{equation*}
\int_{\mathcal{H}^{1,0}}L_1(y)\chi_R(\|y\|_{1,0})\mu_k(dy)\leq \frac{A_0}{2}.
\end{equation*}
We pass to the limits $k\to\infty$, $R\to\infty$ (in this order, with the use of Fatou's lemma in the second limit), we get
\begin{equation}
\int_{\mathcal{H}^{1,0}}L_1(y)\mu(dy)\leq \frac{A_0}{2}.\label{equal_partiel}
\end{equation}
A similar procedure applied to $(\ref{estim_unif_H2H1})$ and $(\ref{estim_unif_Gp})$ gives $(\ref{inviscid_normH2H1})$ and $(\ref{inviscid_Ep})$. And $(\ref{inviscid_normH2H1})$ implies $(\ref{inviscid_concent}).$ 
\item \textbf{Inviscid limit and its invariance under KG.}
The following diagram is the general scheme of the proof. 
$$
\hspace{10mm}
\xymatrix{
  \d\mu_{k} \ar@{=}[r]^{(I)} \ar[d]^{(III)} & \mu_{k} \ar[d]^{(II)} \\
    \mathfrak{P}_t^*\mu \ar@{=}[r]^{(IV)} & \mu
  }
$$
The point $(I)$ is just the invariance of $\mu_k$ under $\phi_t^k.$ The point $(II)$ is in weak sense. The point $(IV)$ follows immedialety $(III).$ Let us, then, prove $(III)$. Let $f$ be a bounded Lipshitz function on $\mathcal{H}^{1,0}$, suppose, without loss of generality, that $f$ is bounded by $1$ and denote by $C_f$ its Lipschitz constant. We have
\begin{align*}
(\d\mu_k,f)-(\mathfrak{P}^*_{t}\mu,f)&=(\mu_k,\b f)-(\mu,\mathfrak{P}_tf)\\
&=(\mu_k,\b f-\mathfrak{P}_tf)-(\mu-\mu_k,\mathfrak{P}_tf)\\
&=A+B.
\end{align*}

By weak convergence of $\mu_k$ towards $\mu$ as $k\to\infty$, we have that $B\to 0$ as $k\to \infty.$ Now since the measures $\mu_k$ and $\mu$ are concentrated on $\mathcal{H}^{2,1},$ we can restrict the integrals on this space. 
\begin{align*}
|A|&\leq \int_{B_R}\E|f(\phi_t^kw)-f(\phi_tw)|\mu_k(dw)+\int_{\mathcal{H}^{2,1}\backslash B_R}\E|f(\phi_t^kw)-f(\phi_tw)|\mu_k(dw)\\
&=I_1+I_2.
\end{align*}
Recalling that $f$ is bounded by $1$, we use Chebyshev inequality to find
\begin{equation*}
I_2\leq \frac{C}{R^2}.
\end{equation*}
Now
\begin{align*}
I_1 &=\int_{B_R}\E\left[|f(\phi_t^kw)-f(\phi_tw)|\Bbb 1_{\{\|z_\alpha\|_{L^\infty_tH^2}\leq r\sqrt{\alpha}\}}\right]\mu_k(dw)\\ &+\int_{B_R}\E\left[|f(\phi_t^kw)-f(\phi_tw)|\Bbb 1_{\{\|z_\alpha\|_{L^\infty_tH^2}> r\sqrt{\alpha}\}}\right]\mu_k(dw)\\
&=I_1^1+I_1^2.
\end{align*}
By Chebyshev inequality we get
\begin{equation*}
I_1^2\leq \frac{C}{r^2}.
\end{equation*}
Recall that $f$ is Lipschitz, then using  Corollary $\ref{coro}$ we find
\begin{align*}
I_1^1\leq C_f\sup_{w\in B_R}\E\left[\|\phi_t^kw-\phi_tw\|\Bbb 1_{\{\|z_\alpha\|_{L^\infty_tH^2}\leq r\sqrt{\alpha_k}\}}\right]\leq C_{f,R,r}\sqrt{\alpha_k}.
\end{align*}
Now take in the good order the limits $k\to \infty$, $r,R\to\infty$ to finish the argument.
\end{enumerate}
The proof is complete.
\end{proof}

\section{Qualitative properties for the distribution of the Hamiltonian}
The proof of Theorem $\ref{densit}$ below is inspired by the method developped in $\cite{kuksin_nondegeul,armen_nondegcgl,KS12}$, however the general argument is modified because of the lack of conservation laws present in our situation. In the case of the Schrödinger and Euler equations $\cite{armen_nondegcgl,KS12}$, the combination of two conservation laws allowed to control the measure uniformly arround zero, such a control was a useful step in the proof of some absolute continuity properties whose strategy relies in part on a spliting argument. Here, without such uniform control around zero, we show that the final conclusion is still true by using furthermore an approximation argument. 

\begin{thm}\label{densit}
Suppose $a_m\neq 0$ for any $m\geq 0$, then the distribution of the Hamiltonian $E(y)$ under $\mu$ has a density w.r.t. the Lebesgue measure on $\R.$ 
\end{thm}
Before proving the above result let us prove some balance type relations. For a continuous function $h:\R\to\R,$ set
\begin{equation*}
H(x)=\int_0^xh(r)dr.
\end{equation*}
\begin{prop}
Let $\alpha \in (0,1)$ and $\mu_\alpha$ the invariant measure constructed for the problem $(\ref{ddKG})$. Let $h\in C_0^\infty(\R)$, we have
\begin{equation*}
\E_{\mu_\alpha}\left[H(E)(A_0- \|\dt u\|_1^2)\right]+\frac{1}{2}\E_{\mu_\alpha}\left[h(E)\sum_{m}a_m^2(\dt u,e_m)^2\right]=0,
\end{equation*}
where $\E_\mu$ denotes integral with respect to $\mu$ and $E$ is the Hamiltonian of the Klein-Gordon equation $(\ref{KG}).$
\end{prop}
\begin{proof}
Consider the second order linear ODE
\begin{equation}\label{chap5_ODE}
-\Phi''_\lambda+\lambda\Phi_\lambda=h \ \ \ \ \ \lambda \in (0,1),
\end{equation}
with initial data $\Phi(0)=\Phi'(0)=0$. Then it is a matter of direct verification that its solution is
\begin{equation*}
\Phi_\lambda(x)=\frac{1}{2\sqrt{\lambda}}\int_0^x\left(e^{-(x-y)\sqrt{\lambda}}-e^{(x-y)\sqrt{\lambda}}\right)h(y)dy.
\end{equation*}
The good behaviour of $\Phi_\lambda(x)$ at $x\to\infty$ allows to 	apply the Itô formula (Theorem $A.7.5$ and Corollary $A.7.6$ of \cite{KS12}) to $\Phi_\lambda\circ E(y)$: 
\begin{align*}
\Phi_\lambda(E(u))&=\Phi_\lambda(E(u_0))\nonumber\\
&+\int_0^t\left(\Phi_\lambda'(E(u))\left\{(\nabla_yE,g)+\frac{\alpha}{2}\sum_{m}a_m^2(\nabla_y^2E,e_m)\right\}+\Phi_\lambda''(E(u))\sum_{m}a_m^2(\nabla_yE,e_m)^2\right)ds\nonumber\\ 
&+\sum_{m}a_m\int_0^t\Phi_\lambda'(E(u))(\nabla_yE,e_m)d\beta_m(s),
\end{align*}
where
\begin{align*}
g=[g_1,g_2]=[\dt u,\Delta_0 u-u^3+\alpha\Delta_0\dt u].
\end{align*}
Taking the expectation with respect to $\mu_\alpha$ and using the stationarity of the latter, we are led to

\begin{equation}\label{chap5_Balance_lambda}
\E_{\mu_\alpha}\left[\Phi_\lambda'(E(s))\left(A_0-\|\dt u\|_1^2\right)\right]+\frac{1}{2}\E_{\mu_\alpha}\left[\Phi_\lambda''(E(s))\sum_{m}a_m^2(\dt u,e_m)^2\right]=0.
\end{equation}

Now, we have that

\begin{equation*}
\Phi'_\lambda(x)=\frac{-1}{2}\int_0^x\left(e^{-(x-y)\sqrt{\lambda}}+e^{(x-y)\sqrt{\lambda}}\right)h(y)dy,
\end{equation*}
Then we see, using the equation $(\ref{chap5_ODE})$ and the Lebesgue dominated convergence theorem, that, as $\lambda\to 0,$
\begin{align*}
\Phi'_\lambda(x) &\to -\int_0^xh(y)dy=-H(x),\\
\Phi''_\lambda(x) &\to -h(x).
\end{align*}
It remains to apply again the Lebesgue dominated convergence theorem in $(\ref{chap5_Balance_lambda})$ to arrive at the claim.
\end{proof}
By a standard approximation argument one can pass from $C_0^\infty$-functions to indicators on intervals functions, then using the monoton class theorem we arrive at:
\begin{cor}
For any Borel set $\Gamma\subset\R$, we have
\begin{equation}\label{chap5_inequal_effective_Abs_Cont}
\E_{\mu_\alpha}\left[1_{\Gamma}(E)\sum_{m}a_m^2(\dt u, e_m)^2\right]\leq Cl(\Gamma)
\end{equation}
\end{cor}

\begin{proof}[Proof of Theorem $\ref{densit}$]
Thanks to the Portmanteau theorem the proof is restricted to the invariant measures $\mu_\alpha$ associated to the stochastic problem as long as the resulting estimates are uniform in $\alpha$. It then consists of the following two steps:
\\
\textbf{Absolute continuity on the interval $]0,+\infty[.$} By the regularity property, it suffices to consider the intervals $[\epsilon,+\infty[,$ where $\epsilon>0$ is arbitrarily small. Let's define the sets $$I_\epsilon=\{[u,\dt u]\in \mathcal{H}^{2,1}, \ \|\dt u\|\in [\epsilon,+\infty[\}.$$ 
Now write
\begin{align*}
\sum_{m\geq 0}a_m^2(\dt u,e_m)^2 &\geq\sum_{m\leq N}a_m^2(\dt u,e_m)^2\\
&\geq \underline{a}_N^2\sum_{m\leq N}(\dt u,e_m)^2\\
&=\underline{a}_N^2\left(\sum_{m\geq 0}(\dt u,e_m)^2-\sum_{m>N}(\dt u,e_m)^2\right)\\
&\geq \underline{a}_N^2\left(\|\dt u\|^2-(m_0^2+\lambda_N)^{-1}\|\dt u\|_1^2\right),
\end{align*}
where $\underline{a}_N:=\min\{a_m, \ 0\leq m\leq N\}.$
Consider the set
$$I_{\epsilon,R}=\{\|\dt u\|\geq\epsilon,\ \ \|\dt u\|_1\leq R\}\subset I_\epsilon.$$
We have on $I_{\epsilon,R}$
\begin{equation}\label{IepsL}
\sum_{m\geq 0}a_m^2(\dt u,e_m)^2\geq \underline{a}_N^2(\epsilon^2-(m_0^2+\lambda_N)^{-1}R^2):=\frac{1}{C_{N,R,\epsilon}}.
\end{equation}
Remark that, since $\lambda_N\to\infty$, for any $\epsilon>0$, any $R>0$, we can choose $N$ so that $C_{N,R,\epsilon}$ be positive. Then combining $(\ref{chap5_inequal_effective_Abs_Cont})$ and $(\ref{IepsL})$, we find
\begin{equation*}
\mu_\alpha(E^{-1}(\Gamma)\cap I_{\epsilon,R})\lleq C_{N,R,\epsilon}l(\Gamma),
\end{equation*}
on the other hand, we have by Chebyshev inequality
\begin{equation*}
\mu_\alpha(E^{-1}(\Gamma)\cap(I_\epsilon\backslash I_{\epsilon,R}))\lleq R^{-2},
\end{equation*}
then, for any $\epsilon>0,$ we have, with use of $(\ref{chap5_inequal_effective_Abs_Cont}),$
\begin{equation*}
\P(E(y)\in\Gamma\cap [\epsilon,\infty))\lleq R^{-2}+C_{N,R,\epsilon}l(\Gamma), \ \ \forall R>0.
\end{equation*}
\textbf{Now we prove that $\P(u\equiv 0)=0.$}
It suffices to show that for some $m$, for any $\alpha>0,$ $\P(u_m=0)=0,$ where $u_m$ is the projection of $u$ on the direction $e_m.$ So consider the projected equation:
\begin{equation*}
y_m(t)=y_m(0)+\int_0^tg_m(s)ds+\hat{\zeta}_m(t),
\end{equation*}
where
\begin{align*}
y_m &=[u_m,\dt u_m],\\
g_m &=[\dt{u}_m,\ (\Delta_0 u-u^3+\alpha\Delta_0\dt u,e_m)],\\
\hat{\zeta}_m(t) &=a_m\beta_m(t).
\end{align*}
An estimate of the form $(\ref{chap5_inequal_effective_Abs_Cont})$ can be derived in a same manner, we use it in the mind of the above procedure. It is clear that the quadratic variation of $u_m$ is bounded from below, it remains to control the drift term. Namely, it suffices to have that $\E(\|[u,\dt u]\|_{2,1}^2+(u^3,e_m))<\infty$ for all $\alpha>0$ to finish the proof, but this is ensured by $(\ref{estim_unif_Gp})$ and $(\ref{estim_unif_H2H1}).$ 
The proof is finished.
\end{proof}

\begin{rmq}
One could derive an inequality of type $(\ref{chap5_inequal_effective_Abs_Cont})$ by using the local time approach. We, first, apply the Itô formula to $E(y)$:
\begin{equation*}
E(y(t))=E(y(0))+\alpha\int_0^t\left(\frac{A_0}{2}-\|\dt u\|_1^2\right)ds+\sqrt{\alpha}\sum_{m\geq 0}a_m\int_0^t(\dt u,e_m)d\beta_m(s).
\end{equation*}
Using the stationarity of $y$, the local time $\Lambda_t(a,\omega)$ of $E(y)$ satisfies
\begin{equation}\label{explicit_LT}
\E \Lambda_t(a)=-\alpha t\E\left[\left(\frac{A_0}{2}-\|\dt u(0)\|_1^2\right)\Bbb 1_{(a,+\infty)}(E(y))\right].
\end{equation}
Now let $\Gamma$ be a Borel set of $\R$, the local time identity $(\ref{intro_lt_identity})$ evaluated to the process $E(y)$ at the function $\Bbb 1_{\Gamma}$ yields
\begin{equation*}
\int_{\Gamma}\Lambda_t(a)da=\alpha\sum_{m\geq 0}a_m^2\int_0^t\Bbb 1_{\Gamma}(E(y))(\dt u,e_m)^2ds.
\end{equation*}
Again, the stationarity of $u$ implies
\begin{equation}\label{identity_LT}
\int_{\Gamma}\E\Lambda_t(a)da=\alpha t\sum_{m\geq 0}a_m^2\E[\Bbb 1_{\Gamma}(E(y))(\dt u,e_m)^2].
\end{equation}
Combining $(\ref{explicit_LT})$ and $(\ref{identity_LT}),$ we get
\begin{equation*}
\E\left[\Bbb 1_{\Gamma}(E(y))\sum_{m\geq 0}a_m^2(\dt u,e_m)^2\right]=\int_\Gamma\E\left[\left(2\|\dt u(0)\|_1^2-A_0\right)\Bbb 1_{(a,+\infty)}(E(y))\right]da,
\end{equation*}
then, with use of $(\ref{estim_unif_L1})$, we find
\begin{equation*}
\E\left[\Bbb 1_{\Gamma}(E(y))\sum_{m\geq 0}a_m^2(\dt u,e_m)^2\right]\leq Cl(\Gamma),
\end{equation*}
where $C$ is a universal constant. We recognize the needed inequality.
\end{rmq}

\begin{prop}
Let $a>1$, set $\sigma=\gamma_0(2aeA_0)^{-1}$, then
\begin{equation}
\E_{\mu}e^{\sigma E(y)}=\int_{\mathcal{H}^{1,0}}e^{\sigma E(y)}\mu(dy)<+\infty.\label{1+delta}
\end{equation}
Consequently, for any $R>0$ we have
\begin{align*}
\P(E(y)\geq R) &\lleq e^{-\sigma R}.\\
\end{align*}
\end{prop}
\begin{proof}
From $(\ref{inviscid_Ep}),$ we write
\begin{align*}
\E_\mu\frac{E^p}{p!}\leq 2\frac{(2pA_0)^p}{\gamma_0^pp!},
\end{align*}
then, with use of the Stirling formula, we get
\begin{align*}
\E_\mu\frac{(\sigma E)^p}{p!}\leq \frac{2p^p}{a^pe^pp!}\sim_{p\to\infty}\frac{\sqrt{2}}{a^p\sqrt{p\pi}}.
\end{align*}
Since $a>1$, the serie of general term $\E_\mu\frac{(\sigma E)^p}{p!}$ is convergent and we get $(\ref{1+delta}).$ Now, we use the Chebyshev inequality to derive the other claim.
\end{proof}

\paragraph{Acknoweledgements.} I thank my advisors Armen Shirikyan and Nikolay Tzvetkov for useful discussions and valuable remarks. I also thank Laurent Thomann for many discussions. This research was supported by the program DIM RDMath of FSMP and Région Ile-de-France.





\bibliographystyle{alpha}
\bibliography{kgg}

\end{document}